\documentclass{amsart}

\usepackage{booktabs}
\usepackage[utf8]{inputenc}
\usepackage[colorlinks=true, linkcolor=black, citecolor=black, filecolor=black, urlcolor=black, pdfborder={0 0 0}, linktoc=all]{hyperref}
\usepackage{amsmath, amsthm, amssymb, mathtools, empheq, cite, enumitem}
\usepackage{todonotes, diagbox, caption, graphicx, pdfpages, pgf, pgfplots, tikz, pgfplots, tikz-cd,pgfmath, mathpazo, ifthen, color}
\usepackage[x11names, table]{xcolor}
\usetikzlibrary{calc, math, intersections, arrows, patterns}
\pgfplotsset{compat=1.8}
\usepackage[framemethod=TikZ]{mdframed}
\usepgfplotslibrary{groupplots}
\makeatletter

\usepackage{subcaption}
\usepackage[capitalise,noabbrev,nameinlink]{cleveref}
\usepackage{csquotes}

\makeatletter
\renewcommand{\p@subfigure}{\thefigure\space(}
\crefformat{subfigure}{#2Figure~#1)#3}
\Crefformat{subfigure}{#2Figure~#1)#3}
\makeatother

\newtheorem{theorem}{Theorem}[section]
\newtheorem{corollary}[theorem]{Corollary}
\newtheorem{proposition}[theorem]{Proposition}
\newtheorem{lemma}[theorem]{Lemma}
\newtheorem{conjecture}[theorem]{Conjecture}
\newtheorem{convention}[theorem]{Convention}

\theoremstyle{definition}
\newtheorem{definition}[theorem]{Definition}
\newtheorem{example}[theorem]{Example}

\newtheorem{remark}[theorem]{Remark}

\newtheorem{observation}[theorem]{Observation}
\newtheorem*{repcorollary}{Corollary \ref{cor_shellable}}

\title[M. Müller]{A combinatorial model for the canonical join complex of alt \texorpdfstring{$\nu$-}-Tamari lattices}
\author[M.~M\"uller]{Matthias M\"uller}
\address[M.~M\"uller]{Institute of Geometry, University of Technology Graz, Austria}
\email{matthias.mueller@tugraz.at}

\newcommand{\defn}[1]{{\color{green!50!black}\textbf{\emph{#1}}}}

\newcommand{\Tam}[1]{\operatorname{Tam}({#1})}
\newcommand{\altTam}[2]{\operatorname{Tam}_{#1}(#2)}
\newcommand{\altDyck}[2]{\operatorname{Dyck}_{#1}(#2)}
\newcommand{\Dyck}[1]{\operatorname{Dyck}({#1})}

\begin{document}

\begin{abstract}
Alt $\nu$-Tamari lattices constitute a remarkable family of lattices associated with lattice paths that broadly generalize the Dyck and Tamari lattices. To systematically study the structural properties of this family, we introduce a combinatorial model that realizes the canonical join complex of alt $\nu$-Tamari lattices. Serving as a universal tool, this model allows us to prove vertex decomposability, establish an explicit shelling order, and reveal the underlying homology of the canonical join complex of alt $\nu$-Tamari lattices.

%Alt $\nu$-Tamari lattices, introduced by Ceballos and Chenevière, constitute a remarkable family of lattices associated with lattice paths that broadly generalize the Dyck and Tamari lattices. To systematically study the structural properties of this family, we introduce a combinatorial model that realizes the canonical join complex of alt $\nu$-Tamari lattices. Serving as a universal tool, this model allows us to prove vertex decomposability, establish an explicit shelling order, and reveal the underlying homology of the canonical join complex of alt $\nu$-Tamari lattices.

%Finally, we use it to reveal the underlying homology of the canonical join complex of alt $\nu$-Tamari lattices.

%Alt $\nu$-Tamari lattices, introduced by Ceballos and Chenevière, constitute a remarkable family of lattices associated with lattice paths that broadly generalize the Dyck and Tamari lattices. To systematically study the structural properties of this family, we introduce a combinatorial model centered around the canonical join complex, introduced by Reading and Barnard as the simplicial complex whose faces correspond to unique "lowest" join representations of lattice elements. Our model bridges the gap between these lattices and its canonical join complex. This connection makes our model a universal tool for studying this family. Finally, we use it to reveal the underlying homology of the canonical join complex of alt $\nu$-Tamari lattices.

\end{abstract}

\maketitle
\tableofcontents
\section{Introduction}
For a fixed positive integer $n$, the classical Dyck and Tamari lattice \cite{Tamari51, Sanley1975} are partial orders on Dyck paths of semilength $n$. The former is also frequently referred to as the Stanley lattice. Equivalently, a Dyck path may be viewed as a lattice path from $(0,0)$ to $(2n,0)$ using unit east $E$ and north $N$ steps staying above the main diagonal. Given a fixed northeast path $\nu$, Préville-Ratelle and Viennot introduced the \defn{$\nu$-Tamari lattice}, along with the associated \defn{$\nu$-Dyck lattice}, as natural generalizations to the set of lattice paths lying weakly above $\nu$, see \cite{preville_vTamari_2017}.
The choice $\nu=(NE)^n$ recovers the classical case, while $\nu=(NE^m)^n$ yields the \defn{$m$-Tamari} lattice, introduced by Bergeron and Préville-Ratelle \cite{Bergeron2012} to state conjectural combinatorial interpretations for the dimensions of certain spaces arising in the study of trivariate diagonal harmonics.
A further generalization of the $\nu$-Tamari lattice, known as the \defn{alt~$\nu$-Tamari} lattice $\altTam{\nu}{\delta}$, was recently introduced by Ceballos and Chenevière in \cite{CCh2024}. This is a family of lattices, depending on an additional parameter $\delta$. Depending on a particular choice of $\delta$, the alt $\nu$-Tamari lattice is the $\nu$-Tamari and~$\nu$-Dyck lattice. Moreover, the authors establish that all alt $\nu$-Tamari lattices share the same number of linear intervals, revealing a first invariance across this entire family.
%***************************************

To further investigate the structural properties of alt $\nu$-Tamari lattices, we utilize a lattice-theoretic factorization known as the \defn{canonical join representation}. In a finite join-semidistributive lattice $L$, every element admits such a representation, specifically in the case of alt $\nu$-Tamari lattices. The \defn{canonical join complex} of $L$, is the simplicial complex, whose faces are canonical join representations of elements of $L$. 
This approach is rooted in the theory of lattice congruences of the weak order on the symmetric group, pioneered by Reading \cite{Reading2004, reading2015}. In \cite{reading2015}, Reading provided an elegant combinatorial model for the canonical join representation of permutations in terms of non-crossing arc diagrams.
Subsequently, Barnard examined the canonical join complex of the classical Tamari lattice, utilizing the induced complex of right-noncrossing arc diagrams as a combinatorial model for the canonical join complex \cite{barnard2019, Barnard2020}.

%***********************************************************
In Section~\ref{section::boxcomplex}, we introduce the first combinatorial model for the canonical join complex of alt $\nu$-Tamari lattices, which we call the \defn{box complex}. The first part of our work is to establish this combinatorial realization.

%***********************************************************

\begin{theorem}\label{realization}
    The box complex realizes the canonical join complex of the alt~$\nu$-Tamari~lattice.
\end{theorem}
Building on this combinatorial realization, our second main result establishes \defn{vertex decomposability} of this complex. This resolves a conjecture posed by Mühle \cite[Conjecture 3.12]{Muehle2021}, extending his statement from the specialized subclass of $\alpha$-Tamari lattices to the general case of alt $\nu$-Tamari lattices.
\begin{theorem}\label{vertex_decomposable}
The canonical join complex of alt $\nu$-Tamari lattices is vertex decomposable.
\end{theorem}

As a consequence, we obtain  Corollary~\ref{cor_shellable} \cite[Theorem 11.3]{Anders1997}.
\begin{corollary}\label{cor_shellable}
     The canonical join complex of the $\text{alt}$ $\nu$-Tamari lattice is shellable.
\end{corollary}

As another application of the box complex, we compute the \defn{Euler characteristic} of the canonical join complex of alt $\nu$-Tamari lattices. We show that this value is an invariant of all alt $\nu$-Tamari lattices for fixed $\nu$. Specifically, for a fixed path $\nu$, the Euler characteristic is determined by the Narayana polynomial $N_\nu$, see Definition~\ref{definition_nara_poly}.
\begin{proposition}\label{theorem_reciprocity}
Let $\nu$ be a northeast path and increment vector $\delta$. The Euler characteristic of the canonical join complex of any alt $\nu$-Tamari lattice $\altTam{\nu}{\delta}$ is given by:
\[
\chi(\Delta_{\mathrm{Tam}_\nu(\delta)}) = 1 - N_\nu(-1).
\]
In particular, the Euler characteristic is independent of the increment vector $\delta$.
\end{proposition}
By Corollary~\ref{cor_shellable}, the canonical join complex of the alt $\nu$-Tamari lattice is homotopy equivalent to a wedge of spheres of varying dimensions. 
In Section~\ref{section::shellability}, we show that for a fixed path $\nu=NE^{\nu_1}\dots NE^{\nu_n}$, the number of $(n-2)$-dimensional spheres is the same for all alt~$\nu$-Tamari lattices (\Cref{thm_hf_invariance}). Moreover, we provide a counting formulation for the case where $\nu_i \geq 2$ for all $i$.
%In Section~\ref{section::shellability}, we show for a fixed path $\nu=NE^{\nu_1}\dots NE^{\nu_n}$, the number of $(n-2)$-dimensional spheres is the same for all alt~$\nu$-Tamari lattices (\cref{thm_hf_invariance}). Moreover, we provide a counting formulation for the case where $\nu_i \geq 2$ for all $i$. 
This is made precise in \Cref{theorem_wedge}. To this end, we identify the northeast path $\nu=NE^{\nu_1}\dots NE^{\nu_n}$ with the sequence~$\nu=(\nu_1,\dots,\nu_{n})$ and define the \defn{shrunken path} $\bar{\nu}$ as the path corresponding to the sequence $\bar{\nu}=(\nu_1-2,\dots,\nu_n-2)$, see Definition~\ref{def_shrunked}.

\begin{theorem}\label{theorem_wedge}
 Let $\nu $ be a finite northeast path with $\nu_i \geq 2$. The canonical join complex of all alt $\nu$-Tamari lattices is homotopy equivalent to a wedge of spheres with top dimension~$(n-2)$ and the number of~$(n-2)$-spheres is given by the number of $\bar{\nu}$-Dyck paths.
 In particular, the number of top-dimensional spheres is independent of the increment vector~$\delta$. 
 %In particular, this number is independent of $\delta$.
    %The canonical join complexes of the alt $\nu$-Tamari lattice $\altTam{\nu}{\delta}$ is homotopy equivalent to a wedge of finitely many spheres. Moreover, for all alt $\nu$-Tamari lattices, with $\nu_i\geq 2$ the number of~$n$~-~dimensional spheres is independent of $\delta$ and given by the Fuss-Catalan number$$ \frac{1}{(m-2)n+1} \binom{(m-1)n}{n}. $$
\end{theorem}

Notably, for~$\nu=(NE^m)^n$, the top-dimensional spheres are enumerated by Fuss-Catalan numbers and $m=3$ results in classical Catalan numbers, see Corollary~\ref{corollary_wedge}.
Finally, we consider in~\Cref{section::homology} the homotopy type of the canonical join complex of alt $\nu$-Tamari lattices. However, alt $\nu$-Tamari lattices do not have the same homotopy type, see Example~\ref{example_homotopy}. 

The class of alt $\nu$-Tamari lattices appears to exhibit many shared structural coherences, and our combinatorial model provides a universal tool for their study. Preliminary findings suggest significant correspondences with rowmotion operators within the framework of dynamical algebraic combinatorics. Our combinatorial model serves as the primary vehicle for investigating these relationships \cite{rowmotion}.

%All alt $\nu$-Tamari lattices seem to share more structural properties.
%Preliminary evidence already suggests promising links to rowmotion in dynamical algebraic combinatorics. Our combinatorial model will be a fundamental tool  exploration these connections. This is currently underway.

%\begin{enumerate}
%    \item Methods are new
%    \item Alt $\nu$-Tamari lattices have so many coindcidences. In this paper w explore further on the level of the canonical join complex.
%    \item We give a combinatorial realizeation of the canonical join cmplex of alt $\nu$-Tamari lattices, the box complex.
%    \item We proof vertex decomposability.
%    \item The Euler characterisitc of the canonical join complex $\Delta$ of the alt $\nu$-Tamari lattice ${\altTam{\nu}{\delta}}$ is independent of $\delta$ and given by the Narayana polynomial, evaluated at $-1$, meaning $\chi(\Delta)= -N_\nu(-1)+1$.
%    \item However, alt $\nu$-Tamari lattices do not have the same homotopy type. But they are all homotopically equivalent to a wedge of spheres.
%    \item In special case of $m$-Dyck and $m$-Tamari the number of highest dimensional spheres are equal and given by the Fuss Catalan number. Meaning that their homology groups have the same rank and are of the same order.

%    \end{enumerate}

\section{The canonical join complex of alt \texorpdfstring{$\nu$}{nu}-Tamari lattices}
In this section, we introduce a simplicial complex, called the \defn{box complex} $\Delta_u$. Furthermore, we study the canonical join complex of the alt~$\nu$-Tamari lattice and establish that it is realized by $\Delta_u$ (\Cref{realization}).

\subsection{The box complex}\label{section::boxcomplex}
We begin by defining the box complex $\Delta_u$ associated with a unimodal sequence of positive integers $u=(u_1,\dots, u_n)$. Specifically, the sequence satisfies $u_1 \leq \dots \leq u_k \geq \dots \geq u_n$ for some index~$k \in [n]$.

\begin{definition}\label{def_boxcomples}
    \begin{figure}[!b]
    \centering
\includegraphics[width=0.9\textwidth]{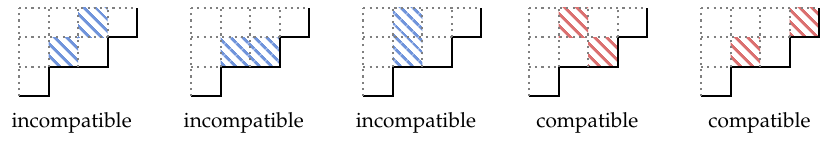}
    \caption{The compatibility relation illustrated for $F_{(3,2,2,1)}$.}
    \label{Fig_v_compatibility}
    \end{figure}
Let $u=(u_1,\dots, u_n)$ be a unimodal sequence of positive integers. We define $F_u$ as a shape consisting of unit boxes arranged such that the $i$th column contains $u_i$ boxes, with all columns top-aligned along a common horizontal line. Two distinct boxes $a, b$ in $F_u$ are said to be \defn{incompatible}, if and only if $a$ is southwest or northeast 
to $b$ and the smallest bounding rectangle containing~$a$ and~$b$ lies entirely within $F_u$ or $a,b$ are in the same column or row. Otherwise, $a$ and $b$ are \defn{compatible} (as illustrated in \Cref{Fig_v_compatibility}).

The \defn{box complex} $\Delta_u$ is defined as the simplicial complex whose faces consist of all subsets of boxes in $F_u$ that are pairwise compatible. Moreover, we define the \defn{transposed shape} $F_{u^t}$ as the shape with $u_i$ boxes in the $i$th row and denote the corresponding box complex by $\Delta_{u^t}$.

\end{definition}
\begin{remark}\label{remark_iso}
We note the following natural isomorphisms regarding the box complex:
\begin{enumerate}
    \item Provided $u_n=1$, we have $\Delta_{(u_1,\dots,u_{n-1},1)} \cong \Delta_{(1,u_1,\dots,u_{n-1})}$.
    \item The box complex $\Delta_u$ is isomorphic to the box complex of its transposed shape, $\Delta_{u^t}$. \label{remark_iso_2}
\end{enumerate}
\end{remark}   
\begin{example}\label{ex_boxcomplex_321}
The box complexes $\Delta_{(3,2,1)}$ and $\Delta_{(1,3,2)}$ are shown in \Cref{Fig_cjc}. As stated in Remark~\ref{remark_iso}, they are isomorphic to each other $\Delta_{(3,2,1)}\cong \Delta_{(1,2,3)}$. Moreover, each box of the shapes corresponds to an unique vertex in the complex. The exact correspondence is illustrated in~\Cref{Fig_ex_cjc_shapes}.
\begin{figure}[!h]
  \centering
  \begin{subfigure}[b]{0.45\textwidth}
    \centering
    \resizebox{1\linewidth}{!}{\includegraphics[page=1]{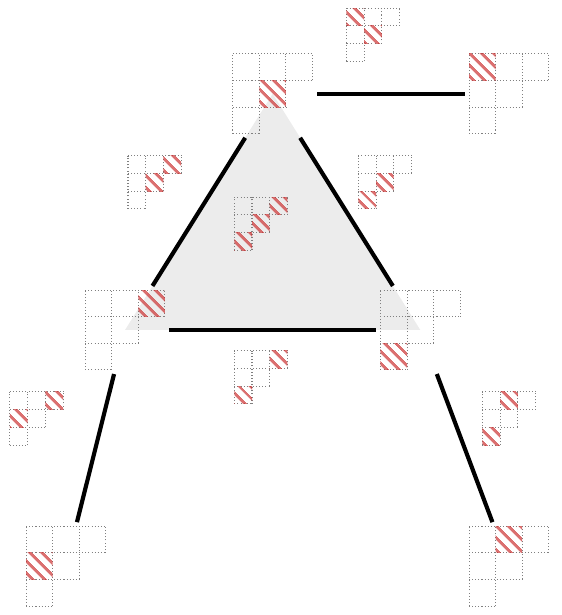}}
    \caption{The box complex $\Delta_{(3,2,1)}$.}
    \label{boxcomplex321}
  \end{subfigure}\hfill
  \begin{subfigure}[b]{0.45\textwidth}
    \centering
    \resizebox{1\linewidth}{!}{\includegraphics[page=1]{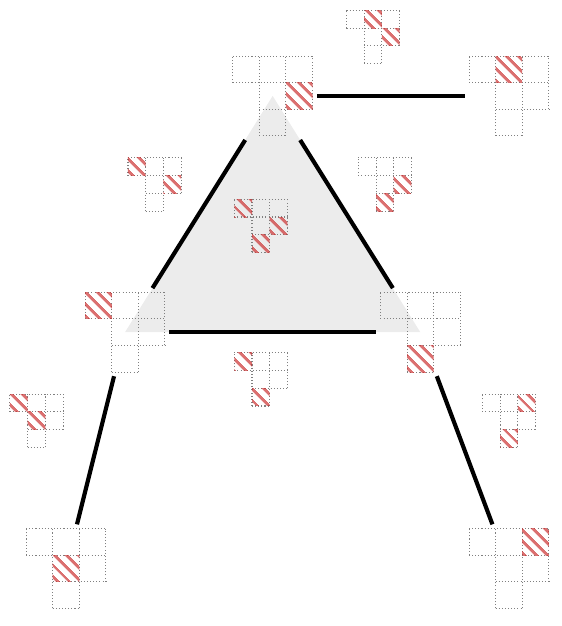}}
    \caption{The box complex $\Delta_{(1,3,2)}$.}
    \label{boxcomplex132}
  \end{subfigure}\hfill
  \caption{The box complexes $\Delta_{(3,2,1)}$ and $\Delta_{(1,3,2)}$ are isomorphic.}
  \label{Fig_cjc}
\end{figure}

\begin{figure}[!h]
    \centering
\includegraphics[width=0.8\textwidth]{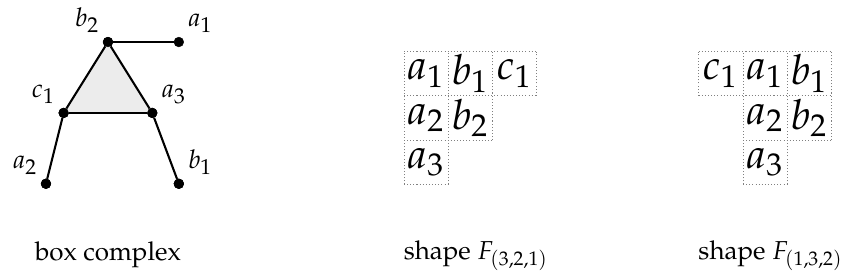}
    \caption{The correspondence between faces in the box complex and boxes in the shape for ${(3,2,1)}$ and ${(1,3,2)}$.}
    \label{Fig_ex_cjc_shapes}
    \end{figure}
\end{example}

\subsection{The canonical join-complex}
In this subsection, we review the construction of the canonical join complex via an edge-labeling. Furthermore, we establish a condition that is particularly well-suited for our subsequent application to the alt~$\nu$-Tamari lattice (Lemma~\ref{observation_labeling}).

A \defn{lattice} is a partially ordered set (short \defn{poset})~$(L, \leq)$ such that any two elements $p, q \in L$ have a unique greatest lower bound, called the \defn{meet}, denoted~$p \wedge q$ and a unique least upper bound, called the \defn{join}, denoted~$p \vee q$. A subset $Q \subseteq L$ is a \defn{sublattice} if it is a lattice under the induced order and is closed under the meet and join operations of $L$, that is, for all $p, q \in Q$, the elements $p \wedge q$ and $p \vee q$ as computed in~$L$ must also lie in $Q$. We denote a cover relation in~$L$ by $x \lessdot y$ and~$\text{Covers}(L)$ denotes the set of all covers in $L$. An element $j$ of a finite lattice is \defn{join-irreducible} if it covers exactly one element, which we denote by $j_\downarrow$. We denote the set of all join-irreducible elements by $\text{JoinIrr}(L)$. A \defn{join-representation} of~$a \in L$ is a subset of join-irreducible elements $A \subseteq L$ such that 
$\bigvee A = a$. We say that $A$ refines $B$ if each $x \in A$ satisfies~$x \le y$ for some $y \in B$. If the join representations of $a$ have a
unique minimal element under refinement, this element is the \defn{canonical join
representation} $\mathrm{Can}(a)$, and its elements are the canonical
joinands of $a$. According to \cite[Proposition 2.2]{reading2015} the set of canonical join representations of $L$ forms a simplicial complex\footnote{We recall several basic notions regarding simplicial complexes in~\Cref{section::Notions_simplicialcomplex}.}, the \defn{canonical join complex}. For more, we refer to \cite{barnard2019}.
A finite lattice $L$ is \defn{join-semidistributive} if, for all $x, y, z \in L$, the following implication holds:$$x\vee y = x\vee z \implies x \vee y = x \vee (y\wedge z).$$
According to \cite[Theorem 2.24]{Freese1995}, $L$ is join-semidistributive if and only if every element possesses a canonical join representation. It is shown in \cite[Lemma 1.8]{Adaricheva2003} that for any cover relation $x \lessdot y$, the set $\{c \in L \mid x \vee c = y\}$ contains a unique, join-irreducible minimal element. This guarantees that the following edge-labeling of a join-semidistributive lattices is well-defined: $$\lambda_{jsd}: \text{Covers}(L) \to \text{JoinIrr}(L), \quad (x,y) \mapsto \min \{c \in L \mid x \vee c = y\}.$$
In terms of this labeling, the characterization in \cite[Lemma 19(1)]{barnard2019} identifies the canonical join representation of $a \in L$ as the set of labels of edges covered by $a$: $$\text{Can}(a) = \{ \lambda_{jsd}(a', a) \mid a' \lessdot a \}.$$
An example of a join-semidistributive lattice with edge-labeling~$\lambda_{jsd}$ and the canonical join complex is shown in~\Cref{Fig_sem_lattice_total}. 
\begin{figure}[!h]
  \centering
  \begin{subfigure}[b]{0.45\textwidth}
    \centering
    \resizebox{.4\linewidth}{!}{\includegraphics[page=1]{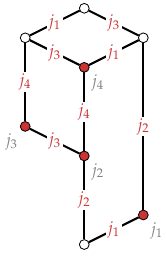}}
    \caption{A join-semidistributive lattice.}
    \label{Fig_sem_lattice}
  \end{subfigure}\hfill
  \begin{subfigure}[b]{0.45\textwidth}
    \centering
    \resizebox{.5\linewidth}{!}{\includegraphics[page=1]{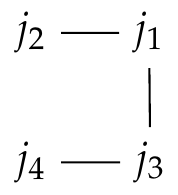}}
    \caption{The canonical join complex of \ref{Fig_sem_lattice}.}
    \label{Fig_sem_cjc}
  \end{subfigure}\hfill

  \caption{A join-semidistributive lattice and its canonical join complex.}
  \label{Fig_sem_lattice_total}
\end{figure}

 Next, we summarize equivalent conditions for the labeling $\lambda_{jsd}$. This looks trivial but will be very useful in \Cref{section_persective_edgelabeling}, considering alt~$\nu$-Tamari lattices.
\begin{lemma}\label{observation_labeling}
Let $L$ be a join-semidistributive lattice, and let $x, y \in L$ such that $x \lessdot y$. For any join-irreducible element $j \in \text{JoinIrr}(L)$ with unique lower cover $j_\downarrow$, the following conditions are equivalent:
\begin{enumerate}
    \item $\lambda_{jsd}(x,y)=j$ \label{cond0}
    \item $\min\{c\in L  \mid x \vee c=y\}=j$ \label{cond1}
    \item $x \wedge j = j_\downarrow$ and $x \vee j = y$\label{cond2} %(or $x \wedge j_\downarrow = x$ and $y \vee j_\downarrow = j$)  
    %\item $y \wedge j_\downarrow = j$ and $x \vee j_\downarrow = x$
    \item $x \wedge j = j_\downarrow$ and $y \wedge j = j$ \label{cond3}
\end{enumerate}
\end{lemma}
\begin{proof}
\eqref{cond0} $\Leftrightarrow$ \eqref{cond1} holds by definition. \eqref{cond1} $\Leftrightarrow$ \eqref{cond2}: This is exactly \cite[Lemma 3.3]{Muehle2023}. \eqref{cond2} $\Rightarrow$ \eqref{cond3}: Follows directly from the absorption law\footnote{ The absorption laws for a lattice $L$ are: $x\wedge(x \vee y)=x$ and $x\vee (x \wedge y)=x$ for all $x,y\in L$.} for a lattice: $y\wedge j = (x\vee j) \wedge j = j$. \eqref{cond3} $\Rightarrow$ \eqref{cond2}: Since~$y \wedge j = j$ holds, $j$ is a lower bound of $y$, $j \leq y$. By $x \wedge j = j ^\downarrow$, $j$ cannot be a lower bound of~$x$, $j \not \leq x$. Hence, $x<x\vee j$ holds. Since $y$ is an upper bound for $j$ this gives~$x<x\vee j \leq x \vee y =y$. Since $x\lessdot y $ we obtain $x\vee j=y$.
\end{proof}

Historically, the relationship captured by Condition~\eqref{cond2} traces back to Jakubík in 1955 \cite{JJ1954}. In this work, two intervals $[x, y]$, $[u, v]$ are defined to be \emph{transposed} if they satisfy $y \wedge u = x$ and $y \vee u = v$ (and symmetrically, $x \wedge v = u$ and $x \vee v = y$). This concept was later formalized and extensively studied in 2010 by Grätzer and Nation \cite{GN2010} under the terminology of \defn{perspectivity}. In light of this, the equivalent properties \eqref{cond0}--\eqref{cond3} characterize this relationship, and covers satisfying them are said to be perspective.

\subsection{The alt \texorpdfstring{$\nu$-}-Tamari lattice}\label{section::altvtamari}
Following the conventions of \cite{CCh2024}, we recall the construction of alt $\nu$-Tamari lattices. Let $\nu$ be a finite northeast path, and let~$F_\nu$ denote its corresponding Ferrers diagram, defined as the region weakly above $\nu$ within its bounding rectangle. We encode~$\nu$ as a sequence of non-negative integers $(\nu_0, \nu_1, \dots, \nu_n)$, where $n$ is the total number of north steps, $\nu_0$ is the number of initial east steps, and each $\nu_i$ denotes the number of east steps immediately following the $i$th north step. We say $\delta = (\delta_1, \dots, \delta_n)$ is an \defn{increment vector} relative to $\nu$ if it satisfies $\delta_i \leq \nu_i$ for all $1 \leq i \leq n$. This vector generates a new path $\check{\nu} = (\check{\nu}_0, \delta_1, \dots, \delta_n)$, where $\check{\nu}_0 = \sum_{i=0}^n \nu_i - \sum_{i=1}^n \delta_i$. Furthermore, let~$\hat{\nu}$ be the northwest lattice path sharing the lowest right corner of $F_{\check{\nu}}$, generated by the step sequence $W^{\nu_0}NW^{\nu_1-\delta_1} \dots NW^{\nu_n-\delta_n}$. Finally, we let $F_{\delta,\nu}$ denote the sub-diagram of $F_{\check{\nu}}$ lying weakly above $\hat{\nu}$, and $L_{\delta,\nu}$ its corresponding set of lattice points. This is illustrated in~\Cref{Fig_altv_path}.
  \begin{figure}[h]
    \centering
    \includegraphics[width=1\textwidth]{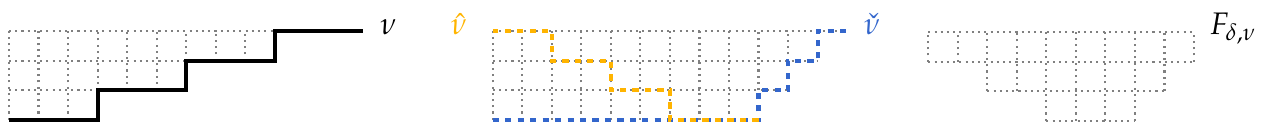}
    \caption{An illustration of the bounding paths constructed from $\nu = (3,3,3,3)$ and $\delta = (1,1,1)$, yielding the step sequences~$\check{\nu} = NE^9NE^1NE^1$ and $\hat{\nu} = W^3NW^2NW^2NW^2$.}
    \label{Fig_altv_path}
\end{figure} 

\begin{definition}[\cite{CCh2024}]
    Two elements $p, q \in L_{\delta,\nu}$ are \defn{$\check{\nu}$-incompatible} if $p$ is strictly southwest (SW) or strictly northeast (NE) of $q$, and the smallest rectangle containing $p$ and $q$ lies completely inside~$F_{\check{\nu}}$. A \defn{$(\delta,\nu)$-tree} is a maximal collection of pairwise $\check{\nu}$-compatible elements in $L_{\delta,\nu}$. Its elements are called \defn{nodes}, and the top-left corner is called the \defn{root}. We associate a rooted binary tree with each~$(\delta,\nu)$-tree $T$ by connecting every $p \in T$ other than the root to the next node in the north or west direction.
\end{definition}

\begin{definition}[alt $\nu$-Tamari lattice~\cite{CCh2024}]\label{def_altvtamari}
  Two $(\delta,\nu)$-trees $T$ and $T'$ are related by a right rotation if $T'$ can be obtained from $T$ by exchanging $q \in T$ with $q' \in T'$, where~$p,r\in T,T'$, and the minimal bounding box containing $p$ and $r$ contains no other nodes different than $q$ or $q'$. The inverse operation is called a left rotation. A node $q'$ in a $(\delta,\nu)$-tree $T$ is called a \defn{descent} if a left rotation can be applied to it. Similarly, it is called an \defn{ascent} if a right rotation can be applied.
  %A node $q'$ in a $(\delta,\nu)$-tree $T$ is called a \defn{descent} if a left rotation can be applied to it. Otherwise, if we can apply a right rotation it is called an \defn{ascent}.
  %\begin{figure}[h]
  %  \centering
  %  \includegraphics[width=0.3\textwidth]{figures/Fig_rotation.pdf}
  %  \caption{Right rotation.}
  %  \label{Fig_cover_relation}
%\end{figure} 
%A node $q$ in a $\nu$-tree $T$ is called a descent, if there exists a node in~$T$ to the south and another to the west of $q$. Equivalently, descents of $T$ are the nodes of $T$ on which we can apply a left rotation. 
The \defn{alt $\nu$-Tamari lattice} $\altTam{\nu}{\delta}$ is the rotation poset of~$(\delta,\nu)$-trees. Examples of the alt $\nu$-Tamari lattices $\altTam{\nu}{\delta}$ are provided in~\Cref{Fig_three_trees}.

\begin{figure}[htb]
    \centering
    \includegraphics[width=1\textwidth]{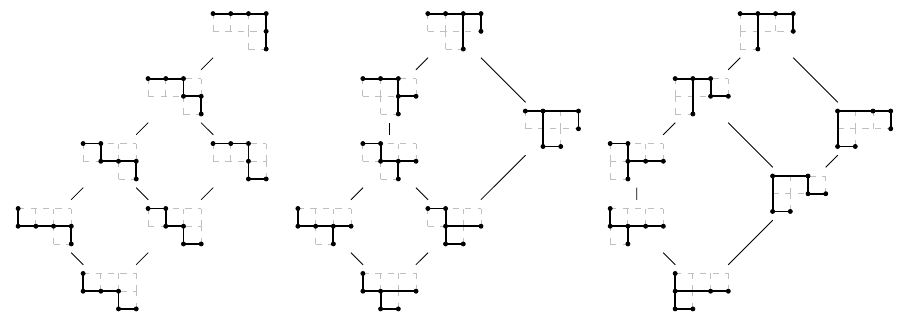}
    \caption{The alt $\nu$-Tamari lattices for $\nu=(1,2,0)$ and from left to right for $\delta_{left}=(0,0)$, $\delta_{middle}=(1,0)$, $\delta_{right}=(2,0)$ via trees.}
    %\caption{Examples of alt $\nu$-Tamari lattices $\altTam{\nu}{\delta}$ for $\nu=ENEEN=(1,2,0)$. Left: the $\nu$-Dyck lattice, for $\delta=(0,0)$. Middle: the lattice for $\delta=(1,0)$. Right: the $\nu$-Tamari lattice, for $\delta=(2,0)$.}
    \label{Fig_three_trees}
\end{figure}
\begin{figure}[htb]
    \centering
    \includegraphics[width=1\textwidth]{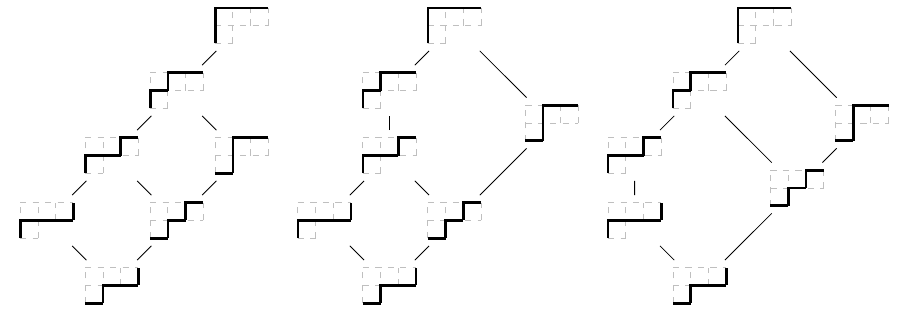}
    %\caption{Examples of alt $\nu$-Tamari lattices $Tam{\nu}({\delta})$ for $\nu=ENEEN=(1,2,0)$. Left: the $\nu$-Dyck lattice, for $\delta=(0,0)$. Middle: the lattice for $\delta=(1,0)$. Right: the $\nu$-Tamari lattice, for $\delta=(2,0)$.}
    \caption{The alt $\nu$-Tamari lattices from \Cref{Fig_three_trees} via paths.}
    \label{Fig_three_paths}
\end{figure}

%\begin{figure}[h!]
%  \centering
%  \begin{subfigure}[b]{0.5\textwidth}
%    \centering
%    \includegraphics[scale=0.55]{figures/Fig_v_lattice.pdf}
%    \caption{$\delta=(1)$ and $\nu=(2,1)$.}
%    \label{Fig_v_lattice}
%  \end{subfigure}\hfill
%  \begin{subfigure}[b]{0.5\textwidth}
%    \centering
%    \includegraphics[scale=0.55]{figures/Fig_altv_lattice.pdf}
%    \caption{$\delta=(1)$ and $\nu=(1,2)$.}
%    \label{Fig_altv_lattice}
%  \end{subfigure}

%  \caption{Alt $\nu$-Tamari lattices for different $\delta$ and $\nu$.}
%  \label{ex_lattices}
%\end{figure}
\end{definition}

\begin{remark}
In Definition~\ref{def_altvtamari}, we recalled that the alt $\nu$-Tamari poset is defined as the rotation poset of $(\delta,\nu)$-trees. Ceballos and Chenevière~\cite{ceballos_vTamari_subword_2020} established a bijection between $(\delta,\nu)$-trees and $\nu$-Dyck paths, where the parameter $\delta$ determines the order relation between the paths. As a small illustrative example, the paths corresponding to the trees in \Cref{Fig_three_trees} are shown in \Cref{Fig_three_paths}. Several special cases for $\delta$ are summarized in Remark~\ref{remark_special_cases_lattice}.
\end{remark}

\begin{remark}\label{remark_special_cases_lattice}
The alt $\nu$-Tamari lattice specializes to several well-known structures depending on the choice of parameters $\nu$ and $\delta$ (where $0\leq i\leq n$ and $1\leq j\leq n$).
\begin{enumerate}[label=(\roman*)]
\item Setting $\nu_i = 1$ and $\delta_j = 1$ recovers the classical Tamari lattice, originally introduced by Tamari in 1962 \cite{Tamari1962}.

\item Setting $\nu_i = 1$ and $\delta_j = 0$ yields the Dyck lattice, also referred to as the Stanley lattice \cite{Sanley1975}.

\item Setting $\nu_i = m$ and $\delta_j = m$ produces the $m$-Tamari lattice, whereas $\delta_j = 0$ yields the $m$-Dyck lattice \cite{Bergeron2012}.

\item Setting $\delta_j = \nu_j$ results in the $\nu$-Tamari lattice \cite{preville_vTamari_2017}. Under these constraints, $(\delta,\nu)$-trees naturally simplify to $\nu$-trees whose nodes occupy $\nu$-compatible positions.
\end{enumerate}
\end{remark}

\begin{remark}
The alt $\nu$-Tamari lattice is uniquely characterized by the shape $F_{\delta,\nu}$. In \cite{alt}, Ceballos defines alt $\nu$-Tamari lattices by specifying the number of points in each column of the shape. Under our notation (Definition~\ref{def_boxcomples}), $u$ instead describes the boxes of the shape. The shape $F_{\delta, \nu}$ corresponds to a unique $F_u$ with $u = u(\delta, \nu)$. We drop the explicit dependence on $(\delta, \nu)$ for simplicity.
\end{remark}

The Tamari lattice is known to be semidistributive, a property inherited by all of its sublattices. Since every alt $\nu$-Tamari lattice forms an interval within a larger Tamari lattice~\cite{preville_vTamari_2017, CCh2024}, it follows that alt $\nu$-Tamari lattices are also semidistributive. Consequently, the canonical join complex of the alt $\nu$-Tamari lattice is well-defined.

%It is known that the Tamari lattice is congruence-uniform, and every sublattice of a congruence-uniform lattice is congruence-uniform as well~\cite[Theorem 4.3]{alanday}. Since every alt $\nu$-Tamari lattice is an interval of a larger Tamari lattice~\cite{preville_vTamari_2017, CCh2024}, alt $\nu$-Tamari lattices are also congruence-uniform and, in particular, join-semidistributive. Therefore, their canonical join complex of the alt $\nu$-Tamari lattice exists.

%\subsection{Join-irreducible elements in alt $\nu$-Tamari lattices}
\begin{lemma}\label{joinirreducible}
    The join-irreducible elements in the alt $\nu$-Tamari lattice are in bijective correspondence with the boxes of the shape.
\end{lemma}
\begin{proof}
A $(\delta,\nu)$-tree $T$ is join-irreducible if it contains a unique descent $p \in T$. Equivalently, $T$ has exactly one non-left-most column with at least two nodes. All other non-left-most columns have exactly one node, and the left-most column can have arbitrary many nodes. Consequently, the descent $p$ is the top-right corner of a box $a$ within the shape~$F_{\delta, \nu}$. Because $T$ is a maximal collection of $\check{\nu}$-compatible elements within $F_{\delta, \nu}$, it must include a node positioned at the bottom-right corner of the box $a$. Thus, we naturally associate the join-irreducible $(\delta,\nu)$-tree $T$ with the unique box $a$.
\end{proof}

\begin{example} \label{ex}
    For $\nu=(1,2)$ and $\delta=(1)$ the four join-irreducible $(\delta,\nu)$-trees correspond to the four boxes of the shape. This is illustrated in~\Cref{Fig_ex_joinirreducibles}.
\end{example}
    \begin{figure}[h]
    \centering
    \includegraphics[width=0.95\textwidth]{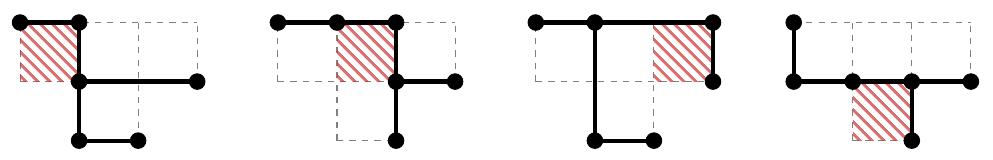} 
    \caption{The join-irreducible $(\delta,\nu)$-trees for $\delta =(1,0)$, $\nu=(1,2,0)$.}
    \label{Fig_ex_joinirreducibles}
    \end{figure}
    
\subsection{The canonical edge-labeling for the alt \texorpdfstring{$\nu$}{nu}-Tamari lattice} \label{section_persective_edgelabeling}
In this subsection, an explicit local description of $\lambda_{jsd}$ for alt $\nu$-Tamari lattices is developed. As a consequence, individual edges can be labeled without reference to the global structure of the lattice. The discussion starts with the special case of the $\nu$-Tamari lattice and is extended to alt $\nu$-Tamari lattices.

\subsubsection{The canonical edge-labeling for the $\nu$-Tamari lattice}
If $(\delta_1,\dots,\delta_n)=(\nu_1,\dots,\nu_n)$, the alt $\nu$-Tamari lattice coincides with the $\nu$-Tamari lattice. In order to obtain a concrete description of $\lambda_{jsd}$ in this setting, we define for any cover $x\lessdot y$ the unique join-irreducible element $j$ satisfying Lemma~\ref{observation_labeling} \eqref{cond3}, that is,
\[
x \wedge j = j_\downarrow \qquad \text{and} \qquad y \wedge j = j.
\]
In this situation, the meet of two $\nu$-trees can be computed efficiently via their bracket vectors.

\begin{definition}[\!\cite{ceballos_vTamari_subword_2020}]
We associate each node of a $\nu$-tree $T$ with its corresponding $y$-coordinate. The bracket vector $b(T)$ is then obtained by performing an \defn{in-order} traversal of $T$, also referred to as symmetric order \cite{preville_vTamari_2017}. Recursively, this traversal visits the left subtree $A$ of a root $x$, followed by $x$ itself, and concludes with the right subtree $B$. A schematic illustration is provided in \Cref{Fig_shema}, which depicts a cover relation $T\lessdot T'$ together with the corresponding bracket vectors~$b(T)=AbBcC$ and $b(T')=AcBcC$. A concrete example of a $\nu$-tree $T$ together with its bracket vector $b(T)$ is shown in~\Cref{Fig_shema_ex}.

\begin{figure}[!h]
    \centering
    \includegraphics[width=0.85\textwidth]{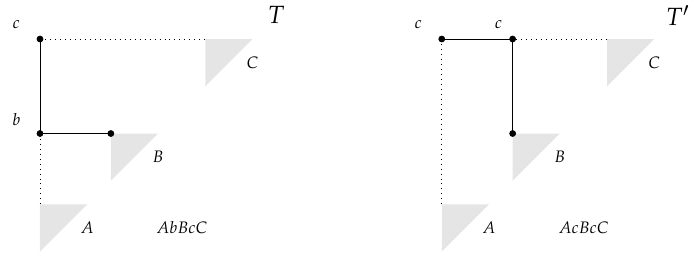} 
    \caption{A cover relation $T \lessdot T'$ between $\nu$-trees, with heights read in in-order \cite[Figure 9]{ceballos_vTamari_subword_2020}.}
    \label{Fig_shema}
    \end{figure}
    \begin{figure}[h]
    \centering
    \includegraphics[width=0.8\textwidth]{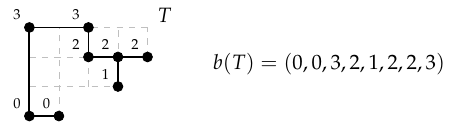} 
    \caption{A $\nu$-tree $T$ and the bracket vector $b(T)$.}
    \label{Fig_shema_ex}
\end{figure}
\end{definition}

According to \cite[Theorem 21]{ceballos_vTamari_subword_2020}, the poset of bracket vectors ordered component-wise is isomorphic to the $\nu$-Tamari lattice. To compute the meet of two $\nu$-trees, we use the following Proposition~\ref{meet_bracketvectors}. 

\begin{proposition}[\textnormal{\!\cite[Proposition 4.12]{ceballos_vTamari_subword_2020}}]\label{meet_bracketvectors}
    The meet of two bracket vectors $b = (b_1,\dots, b_{l})$ and $b' = (b'_1,\dots, b'_{l})$ is their component-wise minimum $$b \wedge b' = (\text{min}\{b_1, b'_1\},\dots, \text{min}\{b_{l}, b'_{l}\}).$$
\end{proposition}

\begin{definition}\label{definition_ji_tree}
For a cover relation $T\lessdot T'$ in the $\nu$-Tamari lattice we define the associated join-irreducible element~$j_{(T,T')}$ as the unique join-irreducible $\nu$-tree, corresponding to the right bottom box of the rectangular shape covered by the rotation from $T$ to $T'$. This is shown in~\Cref{Fig_v_label_withtree}.
     \begin{figure}[h]
    \centering
    \includegraphics[width=1\textwidth]{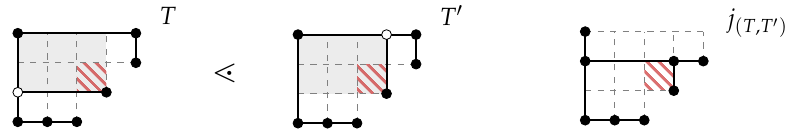} 
    \caption{Cover relation $T\lessdot T'$ and the $\nu$-tree $j_{(T,T')}$.}
    \label{Fig_v_label_withtree}
    \end{figure} 
\end{definition}
\begin{lemma}\label{edgelabeling}
Let $T \lessdot T'$ be a cover relation in the $\nu$-Tamari lattice. Let $j = j_{(T,T')}$ be the associated join-irreducible element and $j_\downarrow$ its unique down cover. Then, $$T \wedge j = {j_\downarrow} \quad \text{ and } \quad  T' \wedge j = j,$$ holds.
\end{lemma}

\begin{proof} 
Let $T \lessdot T'$ be a cover relation in the $\nu$-Tamari lattice with associated join-irreducible element $j = j(T, T')$ as in Definition~\ref{definition_ji_tree}, and unique down cover $j_\downarrow$. By Proposition \ref{meet_bracketvectors}
it is enough to verify for the corresponding bracket vectors:
\begin{align}\label{minimum_bracket_vectors}
\min\{\,b(T),\,b(j)\} = b({j_\downarrow}),
\quad
\min\{\,b(T'),\,b(j)\} = b(j).
\end{align}
For this we split the bracket vector entries into blocks and compare them separately. The $\nu$-tree $T'$ is obtained from $T$ by rotating an unique ascent with $y$-coordinate~$b$ to height~$c>b$, see~\Cref{Fig_shema}. We use $A, B,$ and $C$ to represent the entries of the bracket vector $b(T)$ (resp. $b(T')$) before or after reading the height of a node adjacent to the ascent at height $b$ (resp. descent at height $c$). The bracket vectors~$b(j)$ and $b({j_\downarrow})$ are splitted in a similar way, where~$c=b+1$. We write
\begin{align*}
b(T) &= (A,\;b,\;B,\;c,\;C), \\
b(T') &= (A,\;c,\;B,\;c,\;C),\\
b(j)  &= (A_0,\;b+1,\;B_0,\;b+1,\;C_0),\\
b({j_\downarrow}) &= (A_0,\;b,\;B_0,\;b+1,\;C_0).
\end{align*}
To verify \eqref{minimum_bracket_vectors}, we show:
    \begin{enumerate}[label=(\roman*)]
\item The length of the blocks $A$ and $A_0$ is the same $\ell(A)=\ell(A_0)$.\label{bracket1}
\item The length of the blocks $B$ and $B_0$ is the same $\ell(B)=\ell(B_0)$.\label{bracket2}
\item $A_0$, $B_0$, $C_0$ have minimal entries .
        \label{bracket3}
    \end{enumerate}
    
\ref{bracket1}: We consider the $\nu$-trees $T$ and $j$. By construction, both contain a node at the right bottom corner of the shaded box corresponding to the rotation from $T$ to~$T'$. This is the red shaded box in \Cref{Fig_proof_counting}. Since the bracket vector entries are obtained by reading the $y$-coordiantes of the nodes in in-order, it is enough to consider the shape of the $\nu$-trees $T$ and $j$ left to this unique shaded box. Notice, both trees contain a node in the right bottom corner of this red shaded box. By compatibility, the trees cannot contain a node southwest to this node. This is illustrated in \Cref{Fig_proof_counting} by the gray shaded region. Now, reading the nodes in in-order reads at most one node per column above the gray shaded region, illustrated as colored points. Again, by compatibility, the trees cannot contain a node strictly southwest of any of the marked points. 

Therefore, we can vertically contract the gray area and everything above, keeping the black nodes adjacent to the corresponding ascent or descent. This results in a $\nu$-tree of smaller shape. otherwise, IT we could be possible to add a node in the original shape, contradicting the definition of the original $\nu$-tree. In the $\nu$-tree $j$ we can also contract the gray shaded region and obtain a $\nu$-tree of the same shape. This implies that the umber of nodes is the same.
\begin{figure}[!h]
    \centering
    \includegraphics[width=0.65\textwidth]{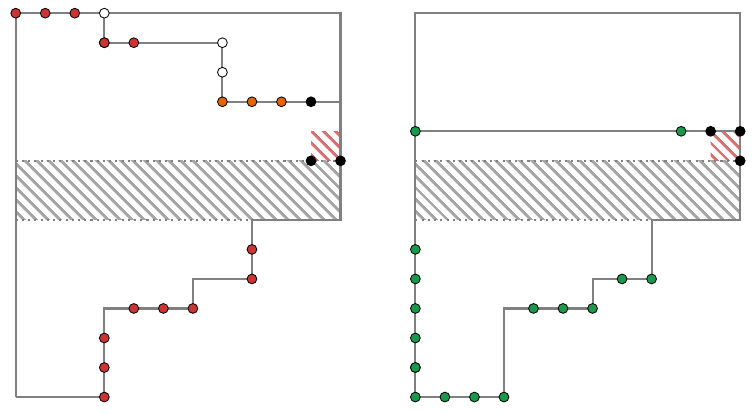} 
    \caption{Left: Shape of $T$ left to shaded box. Right: Shape of~$j$ left to shaded box.}
    \label{Fig_proof_counting}
\end{figure}

\ref{bracket2}: The root $r_B$ of block $B$ is always at the right bottom corner of the shaded box and appears in each of the trees $T$, $T'$, $j$, ${j_\downarrow}$. All nodes southeast to $r_B$ build a subtree of a smaller shape. Since each subtree is a maximal collection of pairwise~$\nu$-compatible elements, $B$ and $B_0$ have the same number of nodes.

\ref{bracket3}: By definition the blocks $A_0$, $B_0$, $C_0$ are characterized by having in each column exactly one node except the first one. Reading this entries in in-order is equal to read the heights of the steps of $\nu$ in order. This are the minimal possible entries. Therefore, we obtain
\[
\min\{\,b(T),\,b(j)\} = (A_0,b,B_0,b+1,C_0)= b({j_\downarrow}),
\]
\[
\min\{\,b(T'),\,b(j)\} = (A_0,b+1,B_0,b+1,C_0)= b(j).
\]
%By  the bracket vectors uniquelly determine the~$\nu$-tree, we conclude\[T \wedge j = j_\downarrow,\quad T'\wedge j = j.\]
\end{proof}
From Lemma~\ref{observation_labeling} and~\ref{edgelabeling}, we deduce the following Corollary~\ref{cor_labeling}.
\begin{corollary}\label{cor_labeling}
    For a cover relation $T\lessdot T'$ in the $\nu$-Tamari lattice, the perspective edge-labeling is given by $j=j(T,T')$.
\end{corollary}
\begin{example}
The $\nu$-Tamari lattice for $\nu=(2,3,0)$ with the perspective edge-labeling is illustrated in~\cref{Fig_ex_lattice}.
        \begin{figure}[!h]
    \centering
\includegraphics[width=0.7\textwidth]{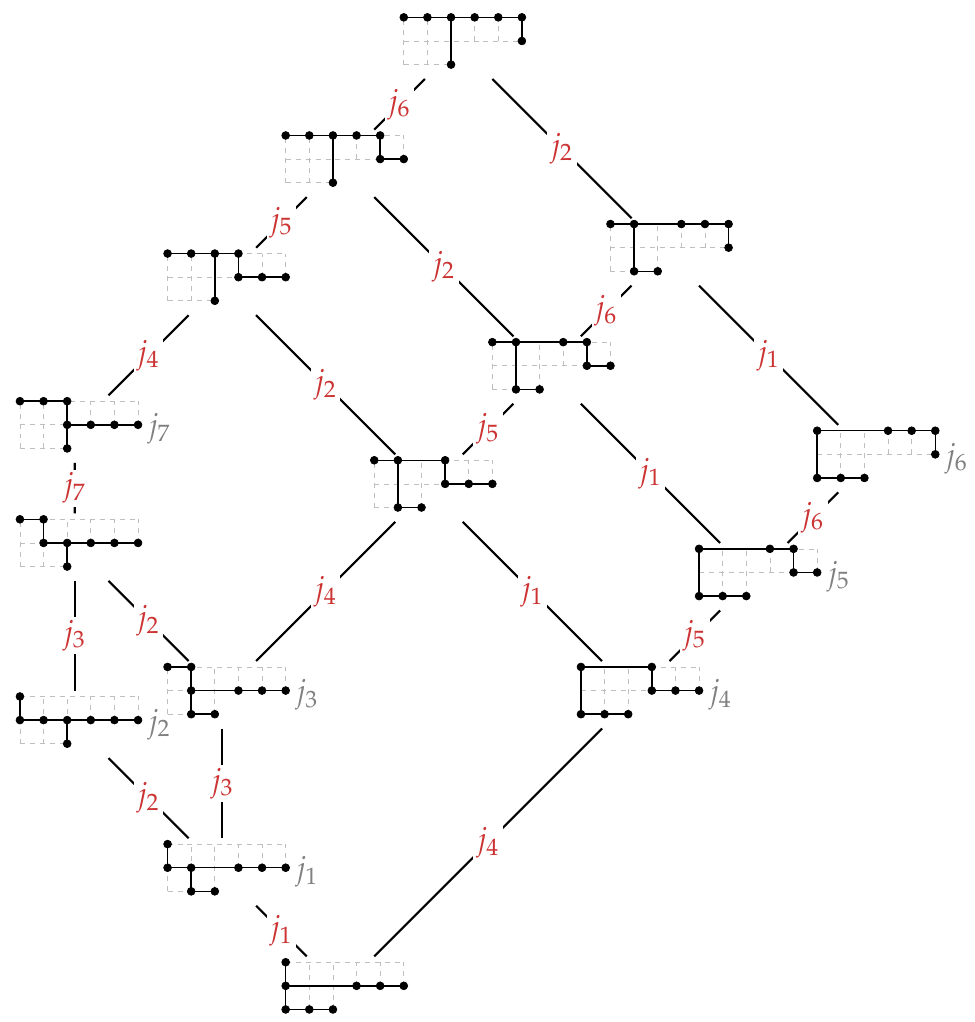} 
    \caption{Alt $\nu$-Tamari lattice for $\nu=(2,3,0)$ and $\delta=(3,0)$.}
    \label{Fig_ex_lattice}
    \end{figure}

\end{example}

%\subsection{The canonical join complex of the alt $\nu$-Tamari lattice}
\subsubsection{The canonical edge-labeling for the alt $\nu$-Tamari lattice}
The alt $\nu$-Tamari lattice is identified as a right interval of the $\nu$-Tamari lattice \cite{CCh2024}. Following Definition~\ref{definition_ji_tree}, the edge labels within a right interval are elements that are join-irreducible in both the $\nu$-Tamari and the alt $\nu$-Tamari lattice. Consequently, the perspective edge-labeling of the $\nu$-Tamari lattice naturally induces a valid perspective edge-labeling on the alt $\nu$-Tamari lattice.
%The alt $\nu$-Tamari lattice is a right interval of the $\nu$-Tamari lattice \cite{CCh2024}. By Definition~\ref{definition_ji_tree}, all the edge labels in an right interval are join irreducible elements that are join irreducible in the $\nu$-Tamari and in the alt $\nu$-Tamari lattice. Therefore, the edge-labeling of the $\nu$-Tamari lattice induces a perspective edge-labeling of the alt $\nu$-Tamari lattice.
\begin{corollary}\label{labeling_altv}
    For a cover relation $T\lessdot T'$ in the alt $\nu$-Tamari lattice, the perspective edge-labeling is given by $j=j_{(T,T')}$.
\end{corollary}

\subsection{Proof of \texorpdfstring{\Cref{realization}}{Theorem 1}} 
Throughout this section, let $\Delta_{(\delta,\nu)}$ denote the box complex of shape $F_{\delta,\nu}$. We establish that the canonical join complex of the alt $\nu$-Tamari lattice $\altTam{\nu}{\delta}$ is isomorphic to $\Delta_{(\delta,\nu)}$. This isomorphism is illustrated in \Cref{fig_comparison_complexes}.

%Fig_ex_alt_lattice
\begin{figure}[h!]
  \centering
  \begin{subfigure}[b]{0.3\textwidth}
    \centering
    \resizebox{1\linewidth}{!}{\includegraphics[page=1]{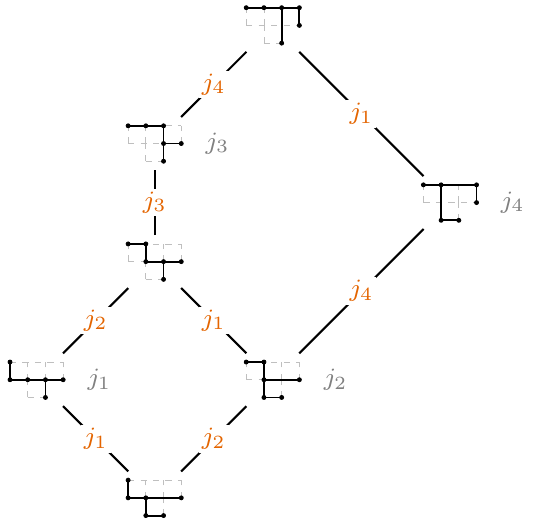}}
    \caption{Alt $\nu$-Tamari lattice for  $\nu=(1,2,0)$ and $\delta=(1,0)$}
    \label{Fig_complex_altv}
  \end{subfigure}\hfill
  \begin{subfigure}[b]{0.3\textwidth}
    \centering
    \resizebox{.8\linewidth}{!}{\includegraphics[page=1]{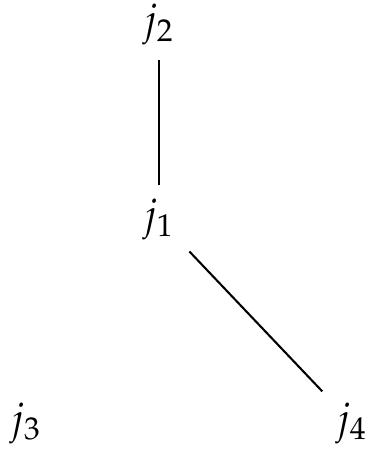}}
    \caption{Canonical join complex of \Cref{Fig_complex_altv}}
    \label{fig:sub:two}
  \end{subfigure}
\hfill
  \begin{subfigure}[b]{0.3\textwidth}
    \centering
    \resizebox{1\linewidth}{!}{\includegraphics[page=1]{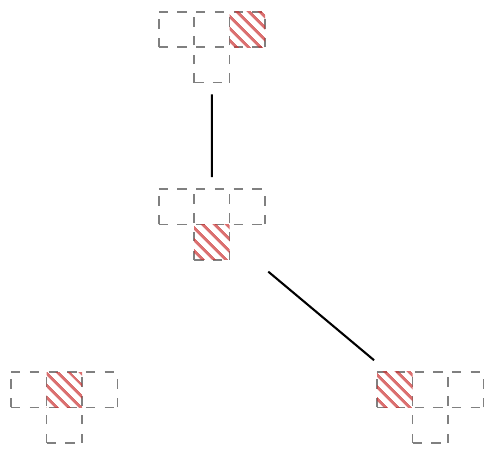}}
    \caption{Box complex for \cref{Fig_complex_altv}}
    \label{Fig_complex_boxes}
  \end{subfigure}
  \caption{The alt $\nu$-Tamari lattice, its canonical join complex, and the box complex for the shape $F_{(1,2,1)}$.}
  \label{fig_comparison_complexes}
\end{figure}

\begin{definition}
    Let $\tau$ be the map from the canonical join complex of $\altTam{\nu}{\delta}$ to the box complex $\Delta_{(\delta,\nu)}$ defined by
    $$\tau:\text{canonical join complex of $\altTam{\nu}{\delta}$} \longrightarrow \text{box complex $\Delta_{(\delta,\nu)}$},$$
    where for a face $F$ of the canonical join complex, $\tau(F)$ is the collection of boxes corresponding to the vertices of $F$ under the bijection established in Lemma~\ref{joinirreducible}.
\end{definition}
\begin{lemma}
    Let $F$ be a face of the canonical join complex of the alt $\nu$-Tamari lattice~$\altTam{\nu}{\delta}$. Then $\tau(F)$ is a face of the box complex $\Delta_{(\delta,\nu)}$. Thus, $\tau$ is well-defined.
\end{lemma}
\begin{proof}
Let $F$ be a face of the canonical join complex of~$\altTam{\nu}{\delta}$. We consider the collection of boxes associated to the vertices of $F$. We show that this collection of boxes $\tau(F)$ is pairwise in compatible position. Since $F$ is the canonical join representation of a $(\delta,\nu)$-tree $T$, each vertex of $F$ corresponds to a label $\lambda_{jsd}(T',T)$ of a down cover $T'\lessdot T$. Assume toward a contradiction that we obtain two boxes~$a$,~$b$ in incompatible position, as in~\Cref{Fig_contradiction_v_comp}.
    \begin{figure}[!h]
    \centering
\includegraphics[width=0.7\textwidth]{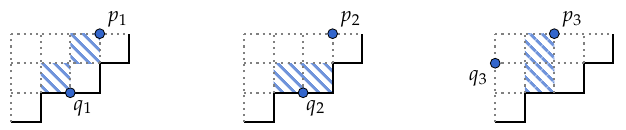} 
    \caption{Pairs of incompatible boxes.}
    \label{Fig_contradiction_v_comp}
    \end{figure}

\begin{enumerate}
\item If \(a\) lies northeast or southwest of \(b\) and the minimal rectangle containing~\(a\) and \(b\) is contained in the shape (respectively, if \(a\) and \(b\) are horizontally aligned), then there exists a node \(q_1\) (respectively, \(q_2\)) at the lower-right corner of the left box and a node \(p_1\) (respectively, \(p_2\)) located above the lower-right corner of the right box. In either case, the smallest rectangle containing the nodes \(p_1\) and \(q_1\) (respectively, \(p_2\) and \(q_2\)) is inside the shape, which contradicts the construction of a \((\delta,\nu)\)-tree. See \Cref{Fig_contradiction_v_comp} (left and middle).

\item If \(a\) and \(b\) are vertically aligned, then there exists a node \(p_3\) above the lower-right corner of the boxes. Since the lower box is shaded, there is also a node \(q_3\) at some horizontal level between $a$ and $b$. Consequently, the minimal rectangle containing \(p_3\) and \(q_3\) lies entirely inside the shape, which contradicts the construction of a \((\delta,\nu)\)-tree. See \Cref{Fig_contradiction_v_comp} (right).

\end{enumerate}
\end{proof}

\begin{definition}
    Let $\theta$ be the map from the box complex $\Delta_{(\delta,\nu)}$ to the canonical join complex of $\altTam{\nu}{\delta}$ defined by
    $$\theta : \text{box complex $\Delta_{(\delta,\nu)}$}\longrightarrow \text{canonical join complex of $\altTam{\nu}{\delta}$},$$
    where for any face $F$ of the box complex, $\theta(F)$ is the set of join irreducible elements corresponding to the vertices of $F$ under the bijection established in Lemma~\ref{joinirreducible}.
\end{definition}

\begin{lemma}
    Let $F$ be a face of the box complex $\Delta_{(\delta,\nu)}$. Then $\theta(F)$ is a face of the canonical join complex of $\altTam{\nu}{\delta}$. Thus, $\theta$ is well-defined.
\end{lemma}
\begin{proof}
Let $F$ be a face of the box complex $\Delta_{(\delta,\nu)}$. Its vertices, $\{v_1, \dots, v_\ell\}$, form a collection of pairwise compatible boxes in $F_{\delta,\nu}$. We refer to them as shaded boxes. According to Lemma~\ref{joinirreducible}, we identify these vertices with the set of join-irreducible~$(\delta,\nu)$-trees $\{T_1, \dots, T_\ell\}$, where each tree $T_i$ possesses a unique descent corresponding to the box $v_i$.

To show that \(\{T_{1},\dots,T_{\ell}\}\) forms a face of the canonical join complex, we construct a \((\delta,\nu)\)-tree \(T\) as follows:
Place nodes in the shape from right to left and from bottom to top using the following rule: 
%Place a node as bottom-most as possible. 
If the box northwest to it is not contained in the shape (resp. a box in the column left to it is shaded) continue placing nodes in a compatible way upward as far as we have placed a node (resp. a node at least at the top horizontal line of the shaded box).

Since the boxes $\{v_1, \dots, v_\ell\}$ are pairwise compatible, we have placed a node at each right bottom corner of a shaded box. Following our construction, there is exactly one node placed above that node. This is because we could place at least one node in the top horizontal line of the shape. Note, if the box northwest to a node is not inside the shape, we do not create a descent corresponding to that box by proceeding placing nodes upwards. Hence each descent of the constructed tree~$T$ corresponds to a shaded boxes.

    Therefore, this construction yields a valid \((\delta,\nu)\)-tree $T$, with down covers labeled by~\(\{T_{1},\dots,T_{\ell}\}\). Hence, \(\{T_{1},\dots,T_{\ell}\}\) is the canonical join decomposition of the constructed $(\delta,\nu)$-tree $T$. This completes the proof.\end{proof}

\begin{example}
Two examples of the construction of a \((\delta,\nu)\)-tree from a given set of pairwise compatible boxes are shown in \Cref{Fig_boxes_tree}.
    \begin{figure}[!h]
    \centering
    \includegraphics[width=.6\textwidth]{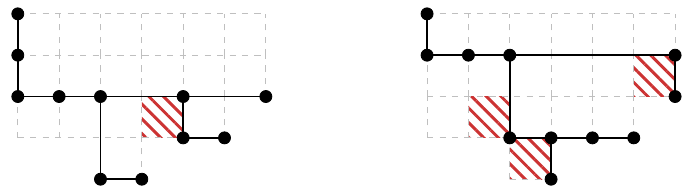} 
    \caption{Two examples of the construction of a \((\delta,\nu)\)-tree from a collection of pairwise compatible boxes.}
    \label{Fig_boxes_tree}
    \end{figure}

%\begin{figure}[!h]
%    \centering
%    \includegraphics[width=.9\textwidth]{figures/Fig_boxes_treealt.pdf} 
%    \caption{Two examples of the construction of a \((\delta,\nu)\)-tree from a collection of pairwise compatible boxes.}
%    \label{Fig_boxes_treealt}
%    \end{figure}
\end{example}

\begin{proof}[Proof of \Cref{realization}]
We show the maps $\tau$, $\theta$ are inverse to each other
    $$\theta \circ \tau=\text{id}_c,\qquad \qquad \qquad  \tau \circ \theta = \text{id}_b,$$ where $\text{id}_c$ is the identity map on the canonical join complex of $\altTam{\nu}{\delta}$ and $\text{id}_b$ the identity on the box complex $\Delta_{(\delta,\nu)}$.
Let $\textbf{c}$ be a face of the canonical join complex and $\textbf{b}$ a face of the box complex.
$\theta ( \tau (\textbf{c}))=\textbf{c}$: Let $T = \bigvee \mathbf{c}$. The boxes utilized in the construction of $\theta$ are precisely the labels of the down-covers of $T$ which, by construction, coincide with the vertices of the face $\tau(\mathbf{c})$ in the box complex.
$\tau ( \theta(\textbf{b}))=\textbf{b}$: The $(\delta,\nu)$-tree $\bigvee \theta(\mathbf{b})$ possesses a unique descent for each vertex in~$\mathbf{b}$, such that the labels of its down-covers are precisely the vertices of $\textbf{b}$.
\end{proof}

%\begin{remark}
%Following \Cref{realization} and Remark~\ref{remark_iso}~\eqref{remark_iso_2}, the box complex $\Delta_u$ serves as a combinatorial model for the canonical join complexes of several prominent lattices:
%\begin{enumerate}[label=(\roman*)]
%\item Setting $u=(n,n-1,\dots,1)$ recovers the classical Tamari lattice.
%\item Setting $u=(1,2,\dots,n)$ yields the Dyck lattice.
%\item Setting $u=(mn,m(n-1),\dots,m)$ produces the $m$-Tamari lattice, whereas setting $u=(m,2m,\dots,mn)$ yields the $m$-Dyck lattice.
%\item Setting $u=(u_1,\dots,u_n)$ with $u_i=\sum_{j=1}^{n-i+1} \nu_i$ generates the $\nu$-Tamari lattice, corresponding to a finite sequence of unit north and east steps $\nu=NE^{\nu_1} \dots NE^{\nu_n}$.
%\end{enumerate}
%\end{remark}

\section{Vertex decomposability}
%This section establishes the vertex decomposability of the canonical join complex for the alt~$\nu$-Tamari lattice, identified here with its box complex realization, see \Cref{realization}. While our approach follows \cite[Section 4.2]{Barnard2020}, our combinatorial model provides new methods to generalize the classical Tamari case.

This section demonstrates that the canonical join complex of the alt~$\nu$-Tamari lattice is vertex decomposable. Our approach follows \cite[Section 4.2]{Barnard2020} and uses new methods arising from the box complex, introduced in \Cref{section::boxcomplex}.
%Throughout, we identify this complex with its combinatorial realization, the box complex, as justified in~\Cref{realization}.

\subsection{Basic notions: vertex decomposability}\label{section::Notions_simplicialcomplex}
In this subsection we provide some necessary background regarding simplicial complexes and
related notions.

A \defn{simplicial complex} $\Delta$ on a finite ground set $V$ is a family of subsets of $V$ that is closed under inclusion, meaning that whenever $\sigma \in \Delta$ and $\tau \subseteq \sigma$, one also has~$\tau \in \Delta$. The elements of $\Delta$ are called \defn{faces}. %It is not assumed that~$\{v\} \in \Delta$ for every~$v \in V$. 
The \defn{vertices} of $\Delta$ are precisely those elements~$v \in V$ for which~$\{v\} \in \Delta$. Consequently, the vertex set of $\Delta$ may be strictly smaller than the ground set~$V$. The \defn{join} of two simplicial complexes~$\Delta_1$ and~$\Delta_2$ is the complex $\Delta_1 \ast \Delta_2 = \{ F_1 \cup F_2 : F_1 \in \Delta_1 \text{ and } F_2 \in \Delta_2 \}$. Given a subset~$V' \subseteq V$, the \defn{induced subcomplex} of $\Delta$ on $V'$ is defined by
$ \Delta_{V'} = \{ F \in \Delta : F \subseteq V' \}$. A \defn{facet} of $\Delta$ is a face that is maximal with respect to inclusion. The dimension of a face~$F \in \Delta$ is defined as $\dim(F) = |F| - 1$. The \defn{dimension} of the complex is then given by~$ \dim(\Delta) = \max \{ \dim(F) : F \in \Delta \}$. A facet is said to be \defn{full-dimensional} if its dimension coincides with $\dim(\Delta)$. The complex $\Delta$ is called \defn{pure} if all of its facets have the same dimension. %cardinality. 
For a face~$F \in \Delta$, the \defn{link} and the \defn{deletion} of~$F$ are defined, respectively, as
$
\ell k_{\Delta}(F) = \{ G \in \Delta : G \cap F = \emptyset,\; G \cup F \in \Delta \}
$ and~$
\Delta \setminus F = \{ G \in \Delta : F \not\subseteq G \}.
$

\begin{definition}[\!\cite{Anders1997}]
\label{vd}
A simplicial complex $\Delta$ is said to be \defn{vertex decomposable} if either $\Delta$ is a simplex (including the empty complex) or there exists a vertex $v \in \Delta$ such that:
\begin{enumerate}[label=(\roman*)]
    \item both $\ell k_{\Delta}(v)$ and $\Delta \setminus \{v\}$ are vertex decomposable,\label{vd1}
    \item no facet of $\ell k_{\Delta}(v)$ is a facet of $\Delta \setminus \{v\}$.\label{vd2}
\end{enumerate}
\end{definition}
The property of vertex-decomposability is preserved under the join operation.
\begin{proposition}[\!\textnormal{\cite[Theorem 3.30]{Jonsson2008}}]\label{vd_join}
Let $\Delta_1$ and $\Delta_2$ be vertex decomposable simplicial complexes. Then, the join $\Delta_1 \ast \Delta_2$ is also vertex decomposable.
\end{proposition}

The notion of vertex decomposability was originally introduced by Provan and Billera \cite{Provan1980} in the setting of pure simplicial complexes. In the pure case, condition~\ref{vd2} is automatically satisfied, since the link and the deletion have the expected dimensions. A vertex $v$ satisfying condition~\ref{vd2} is called a \defn{shedding vertex} of $\Delta$. If~$v$ satisfies both conditions~\ref{vd1} and~\ref{vd2}, then it is referred to as a \defn{decomposing vertex}.

\subsection{Link and deletion in the box complex}
In this subsection, we characterize the link $\ell k_{\Delta_u}(v)$ and the deletion $\Delta_u \setminus \{v\}$ for a vertex $v \in \Delta_u$. Specifically, Lemma~\ref{lemma_link_iso} demonstrates that the link decomposes into the join of two smaller box complexes.

\begin{observation}\label{observation_boxes_shaded}
Applying \Cref{realization}, we identify both the link and the deletion with specific subsets of boxes in $F_u$. Under the compatibility relation, Definition~\ref{def_boxcomples}, the link $\ell k_{\Delta_u}(v)$ comprises all boxes in $F_u$ that are compatible with $v$, see \Cref{Fig_link}. Similarly, the deletion $\Delta_u \setminus \{v\}$ corresponds to the collection of all compatible subsets of boxes that do not contain $v$, illustrated in~\Cref{Fig_deletion}. Note that in terms of the complex, the link and deletion are the subcomplex induced by the shaded boxes.
\end{observation}
\begin{figure}[h!]
  \centering
  \begin{subfigure}[b]{0.45\textwidth}
    \centering
    \resizebox{1\linewidth}{!}{\includegraphics[page=1]{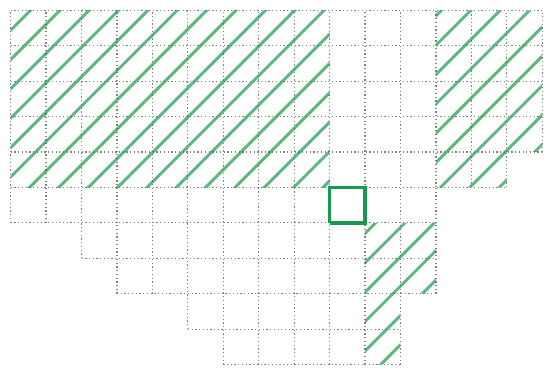}}
    \caption{Link of marked box}
    \label{Fig_link}
  \end{subfigure}\hfill
  \begin{subfigure}[b]{0.45\textwidth}
    \centering
    \resizebox{1\linewidth}{!}{\includegraphics[page=1]{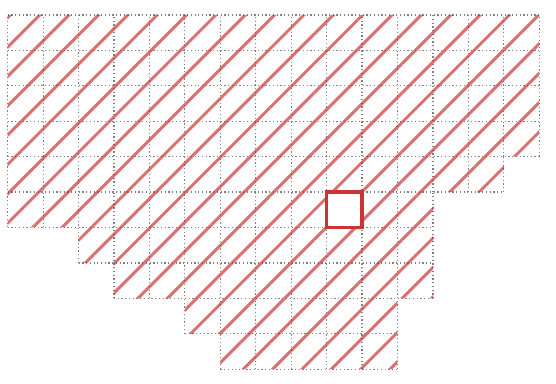}}
    \caption{Deletion of marked box}
    \label{Fig_deletion}
  \end{subfigure}
  \caption{Link and deletion in the shape.}
  \label{Fig_del_link}
\end{figure}
We now take a closer look at the link in the box complex by means of the following definition.

\begin{definition}\label{def_split}
    Let $b$ be a box in the shape $F_u$. We define $U^+_1$ and $U^+_2$ as the sets of compatible boxes northeast and northwest of $b$. We denote by $U^+$ the shape obtained by horizontal glueing of the boxes in $U^+_1$ and $U^+_2$.
The shape $U^-$ is the subshape of compatible boxes southeast of $b$.
\end{definition}
\begin{remark} Formally, Definition~\ref{def_split} can be described as follows: Let $F_u$ be the shape associated with a unimodal sequence $u = (u_1, \dots, u_n)$, Definition~\ref{def_boxcomples}, with rows and columns enumerated top to bottom and left to right. For a box at row~$r$ and column $k$, let $i^* = \max\{ j \ge k : u_j \ge r \}$ if it exists, and $0$ otherwise. We define $u^- = (u_{k+1} - r, \dots, u_{i^*} - r)$ and let $u^+$ be given by $u^+_i = \min\{u_i, r - 1\}$ for~$1 \le i < r$ and $u^+_i = u_{i + i^* - r+1}$ for $i\geq r$. %$r\leq i$.
\end{remark}

Since $u$ is unimodal, this is well defined and $u^+$, $u^-$ are unimodal as well. 

\begin{lemma}\label{lemma_link_iso}
For a vertex $v \in \Delta_u$, the link $\ell k_{\Delta_u}(v)$ is isomorphic to the join~$\Delta_{u^+} \ast \Delta_{u^-}$.
\end{lemma}

\begin{proof}
We identify the link $\ell k_{\Delta_u}(v)$ with the set of boxes in the shape $F_u$ that are compatible with the box corresponding to $v$. By the compatibility relation defined in Definition~\ref{def_boxcomples}, the link can be partitioned into boxes located above and below $v$.
The set of boxes below $v$ corresponds precisely to the subshape $F_{u^-}$, and thus the complex $\Delta_{u^-}$. For the boxes above $v$, we observe that by gluing the regions northeast and northwest of $v$, the original compatibility relations between this regions are preserved. So the boxes above~$v$ can be identified with the shape $F_{u^+}$ and~$\Delta_{u^+}$.

Finally, since every box below $v$ is compatible with every box in the link above~$v$ by compatibility in $F_u$, box complex $\Delta_u$ is isomorphic to the join $\Delta_{u^+} \ast \Delta_{u^-}$.
\end{proof}

\begin{example}\label{ex_shape}
    Consider the shape $F_u$ illustrated in \cref{Fig_del_link}, which is defined by the sequence $u = (6,6,7,8,8,9,10,10,10,10,10,8,5,5,4)$. For the box $v$ located at row 6 and column 10, we obtain the vectors $u^+ = (5,5,5,5,5,5,5,5,5,5,5,4)$ and~$u^- = (4,2)$. By Lemma~\ref{lemma_link_iso}, the link $\ell k_{\Delta_u}(v)$ is isomorphic to~$\Delta_{u^+} \ast \Delta_{u^-}$.
\end{example}

\subsection{Decomposing vertices}\label{decomposing_vertices}
To streamline the proof of vertex decomposability, we define a set of vertices $\{v_1, \dots, v_n\}$ in this subsection. In \Cref{section::vd}, we establish that any ordering of these vertices yields a valid sequence of decomposing vertices.
\begin{definition}\label{def_dec_vertices}
Let $F_u$ be a shape and $q_1$ denote the rightmost vertex in its bottom row. We define $v_1, \dots, v_n$ to be the boxes in $F_u$ that lie in the same row or column as $q_1$.
\end{definition}
\begin{example}
    We continue Example~\ref{ex_shape}. Let $q_1$ be the rightmost box in the bottom row, and denote by $v_1, \dots, v_{13}$ the set of boxes belonging to the same row or column as $q_1$, see \Cref{Fig_dec_vertices}. The deletion $\Delta_u \setminus \{v_1, \dots, v_{13}\}$ is isomorphic to~$\Delta_{u'} \ast \{q_1\}$, with~$u' = (6,6,7,8,8,9,9,9,9,9,8,5,5,4)$. 
\end{example}

\begin{figure}[h!]
  \centering
  \begin{subfigure}[b]{0.45\textwidth}
    \centering
    \resizebox{1\linewidth}{!}{\includegraphics[page=1]{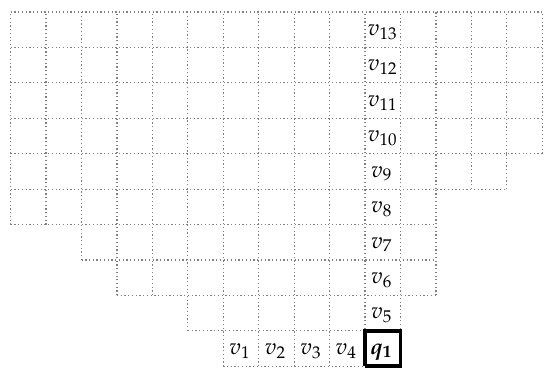}}
    \caption{Decomposing vertices $\{v_1,\dots,v_{13} \}$}
    \label{Fig_dec_vertices}
  \end{subfigure}\hfill
  \begin{subfigure}[b]{0.45\textwidth}
    \centering
    \resizebox{1\linewidth}{!}{\includegraphics[page=1]{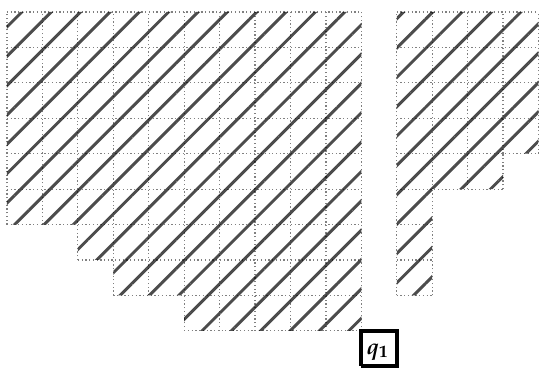}}
    \caption{Deletion of $\{v_1,\dots,v_{13} \}$}
    \label{Fig_dec_deletion}
  \end{subfigure}
  \caption{Decomposing vertices.}
  \label{Fig_dec}
\end{figure}

\subsection{Proof of \cref{vertex_decomposable}}\label{section::vd}
We establish vertex decomposability of $\Delta_u$ by induction on the number of columns, that is the length $n$ of the sequence $u = (u_1, \dots, u_n)$. In the base case $n = 1$, the complex $\Delta_u$ consists of $u_1$ isolated vertices. Such a complex is trivially vertex decomposable. 
For the inductive step, let~$q_1$ be the rightmost vertex in the bottom row and $\{v_1, \dots, v_n\}$ the set of vertices in the same column or row as $q_1$, as specified in Definition~\ref{def_dec_vertices}. We demonstrate that any~$v \in \{v_1, \dots, v_n\}$ is a decomposing vertex for $\Delta_u$, see Definition~\ref{vd}.

%\subsubsection{The deletion $\Delta_u \setminus \{v_1,\dots,v_n\}$ is vertex decomposable:} 
\begin{lemma}\label{lemma_vd_v1tovn}
    The complex $\Delta_u\setminus \{v_1,\dots,v_n\}$ is vertex decomposable.
\end{lemma}
\begin{proof}
Let $F_{u'}$ be the shape obtained by $F_u$ after deleting the row and column of~$q_1$. Then it is enough to argue, that the box complex can be decomposed as join $\Delta_u \setminus \{ v_1,\dots,v_n\} =\Delta_{u'} \ast \{q_1\}$.
Since $q_1$ was chosen rightmost, it is compatible to all boxes northeast to it, since the smallest rectangle including them cannot be inside the shape $F_u$. Moreover, $q_1$ is compatible to all boxes northwest to it, by definition of compatibility. Therefore, the deletion $\Delta_u \setminus\{v_1,\dots,v_n\}$ and so $q_1$ is compatible with all boxes in $$\Delta_u \setminus\{v_1,\dots,v_n\} \cong \Delta_{u'} \ast \{q_1\}.$$ Since joins preserve vertex decomposability by Proposition~\ref{vd_join} and $\Delta_{u'}$ is vertex decomposable by induction, then the claim follows.
\end{proof}

%\subsubsection{The link $\ell k_{\Delta_u}(v)$ is vertex decomposable:}\label{section::link}
\begin{lemma}\label{link}
For any  $J \subseteq [n]$ let $\Delta = \Delta_u \setminus \{v_j : j \in J\}$ and $v \in \{v_j : j \in [n] \setminus J\}$. Then, the link $\ell k_{\Delta}(v)$ is vertex decomposable.
\end{lemma}
\begin{proof}

Two boxes in the same row or column are not compatible. Therefore, the link does not change by removing a box in the same row or column as $q_1$. Hence, $$\ell k_\Delta (v)=\ell k _{\Delta_u}(v).$$ 
By Lemma~\ref{lemma_link_iso}, the link $\ell k_\Delta (v)$ decomposes into the join $\Delta_{u^+} \ast \Delta_{u^-}$ of smaller box complexes, meaning $u^+$ and $u^-$ are of length $<n$. In symbols, $$\ell k _\Delta(v)\cong\Delta_{u^+} \ast \Delta_{u^-}.$$ By the inductive hypothesis, these smaller complexes $\Delta_{u^+}$, $\Delta_{u^-}$ are vertex decomposable. Since by Proposition~\ref{vd_join} joins preserve vertex decomposability, it follows that the link itself is vertex decomposable.

\end{proof}

%\subsubsection{No facet of $\ell k_{\Delta_u}(v)$ is a facet of $\Delta_u \setminus \{v\}$:}\label{section::nofacetoflink}

\begin{lemma}\label{nofacetoflink}
    For any $J \subseteq [n]$ let $\Delta = \Delta_u \setminus \{v_j : j \in J\}$ and $v \in \{v_j : j \in [n] \setminus J\}$. Then, no facet of the link $\ell k _{\Delta}(v)$ is a facet of the deletion~$\Delta \setminus \{v\}$.
\end{lemma}
%Finally, we verify that no facet of the link $\ell k _{\Delta_u}(v)$ is also a facet of the deletion~$\Delta_u \setminus \{v\}$. 
\begin{proof}

Let~$F \in \ell k_{\Delta}(v)$ be a facet of the link of $v$ in $\Delta=\Delta_u \setminus \{v_j : j \in J \subseteq [n]\}$. All boxes comprising the link~$\ell k _{\Delta}(v)=\ell k _{\Delta_u}(v)$ are situated to the northwest or northeast of~$q_1$. Because neither $q_1$ nor any of the boxes within the link have been removed, the union $F \cup \{q_1\}$ forms a valid face in $\Delta$. This guarantees that $F\cup \{q_1\}$ is a face of the deletion $\Delta \setminus \{v_1\}$, meaning no facet of the link can be a maximal face of the deletion.
\end{proof}

\begin{remark}\label{remark_uparrow_boxes}
The box $q_1$ is chosen as the rightmost box in the bottom row. We proceed iteratively with the remaining glued shape, denoting the subsequently chosen elements by $q_2, \dots, q_m$. We point out that the boxes $q_1, \dots, q_m$ can be directly constructed by repeatedly choosing the rightmost box from bottom to top that is compatible with all previously selected boxes, if possible. See \Cref{Fig_termination}.
        \begin{figure}[!h]
    \centering
\includegraphics[width=0.7\textwidth]{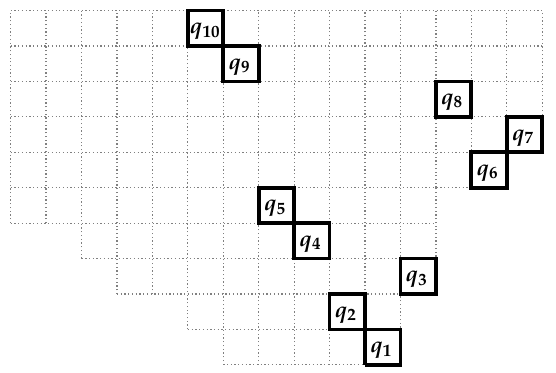} 
    \caption{The vertices $q_1,\dots,q_{10}$ are pairwise compatible.}
    \label{Fig_termination}
    \end{figure}
\end{remark}

%\subsubsection{The deletion $\Delta_u \setminus \{v_k,\dots,v_n\}$ is vertex decomposable:}

\begin{lemma}
    %Let $J \subseteq \{1, \dots, n\}$, $\Delta = \Delta_u \setminus \{v_j : j \in J\}$ and $v \in \{v_j : j \in [n] \setminus J\}$. 
    For any $k\in[n]$, the complex $\Delta_u \setminus \{v_k,\dots,v_n\}$ is vertex decomposable.
\end{lemma}
\begin{proof}
    We argue by induction on $k$. Lemma~\ref{lemma_vd_v1tovn} established the case $k=1$.    
    By induction hypothesis $\Delta_u \setminus \{v_k,\dots,v_n\}$ is vertex decomposable. For the induction step, we argue that $v_k$ is a decomposing vertex in $\Delta'_u=\Delta_u \setminus \{v_{k+1},\dots,v_n\}$. Indeed, the deletion $\Delta'_u\setminus \{v_k\}$ is $\Delta_u \setminus \{v_k,\dots,v_n\}$ and vertex decomposable by induction hypothesis. By Lemma~\ref{link} and \ref{nofacetoflink} the complex is vertex decomposable.
    %This is clear from the arguments above, since two boxes in the same column or row are not compatible, the deletion of such a box, takes no impact.
\end{proof}
    
\section{Applications}
The family of alt $\nu$-Tamari lattices exhibits a remarkable invariance of structural properties. In particular, many fundamental characteristics of $\altTam{\nu}{\delta}$ are independent of the increment vector $\delta$. In this final section, we present some invariants at the level of the canonical join complex.

%In this section, we apply our new combinatorial model, the box complex, to~$\nu$-Dyck and $\nu$-Tamari lattices. 

%This framework enables us to compute the Euler characteristic and investigate the shellability and homology of these structures.

\subsection{The Euler characteristic}\label{section::Euler}
As an initial application, we compute the Euler characteristic of the canonical join complex of alt $\nu$-Tamari lattices $\altTam{\nu}{\delta}$. To do so, we recall the following notions.

\begin{definition}[\textnormal{\cite{tamcom}}]\label{definition_nara_poly}
The \defn{$\nu$-Narayana numbers}, denoted $\text{Nar}_\nu(i)$, count the number of $\nu$-paths containing exactly $i$ valleys $EN$. The~\defn{$\nu$-Narayana polynomial} is given by $$\text{Nar}_\nu(x) = \sum_{i \geq 0} \text{Nar}_\nu(i) x^i.$$\end{definition}
The \defn{$f$-vector} of a simplicial complex $\Delta$ consists of the numbers $f_i(\Delta)$ counting
the faces of dimension $i$. The \defn{Euler characteristic} of $\Delta$ is defined as$$\chi(\Delta) = \sum_{i=0}^{\dim(\Delta)} (-1)^i f_i(\Delta),$$ and the \defn{reduced Euler characteristic} is given by $\tilde{\chi}(\Delta) = \chi(\Delta) - 1$. The following Lemma~\ref{Lemma_f_vector} follows from known facts about unimodular triangulations and is stated in \cite[Section $8$]{alt}.

\begin{proposition}[\textnormal{\cite[Theorem 8.1]{alt}}] \label{Prop_cc_nar}
    The number elements in the alt $\nu$-Tamari lattice $\altTam{\nu}{\delta}$ with $i$ upper covers is given by $Nar_\nu(i)$.
\end{proposition}

\begin{corollary}\label{Lemma_f_vector}
    The number of $(i-1)$-dimensional faces of the canonical join complex $\Delta_{\altTam{\nu}{\delta}}$ is $Nar_\nu(i)$.
\end{corollary}
\begin{proof}
    In any semidistributive lattice, the number of elements with $i$ up-covers is equal to the number of elements with $i$ down-covers \cite[Proposition 9 and Corollary 5]{barnard2019}. Since an $(i-1)$-face of the canonical join complex corresponds to a lattice element with $i$ down covers, the result follows by Proposition~\ref{Prop_cc_nar}.
\end{proof}

Finally, we compute the Euler characteristic of the box complex. In particular, we establish that $N_\nu(-1) = -\widetilde{\chi}(\Delta_{\altTam{\nu}{\delta}})$ holds.

\begin{proof}[Proof of Proposition~\ref{theorem_reciprocity}]
Using Corollary~\ref{Lemma_f_vector}, the number of $(i-1)$-dimensional faces of~$\Delta_{\nu, {Tam}}$ is given by $\text{Nar}_\nu(i)$. Therefore,
    $$\text{Nar}_\nu(-1)= \sum \limits_{i\geq 0} (-1)^i \;\text{Nar}_\nu(i)=\sum \limits_{i\geq 0}  (-1)^i f_{i-1}$$
Since, $f_{-1}=1$, this is equal to   
    $$(-1) \big[ -1 + \sum_{i\geq 0} (-1)^if_i]=-\widetilde{\chi}(\Delta_{\altTam{\nu}{\delta}}).$$
\end{proof}
In particular, the reduced Euler characteristic is given by $\widetilde{\chi}(\Delta_{\altTam{\nu}{\delta}})=-Nar_\nu(-1)$.
%*****************************

\begin{remark}\label{remark_polynomial}
    
For $\nu=(NE^m)^n$, the $\nu$-Narayana numbers simplify to Fuss-Narayana numbers~$N_m(n, k) = \frac{1}{n} \binom{mn}{k} \binom{n}{k-1}$. In this case the $\nu$-Narayana polynomial evaluated at~$-1$ becomes $$\frac{1}{n} \sum \limits_{i=1}^n (-1)^{i}  \binom{mn}{i} \binom{n}{i-1}.$$
Using the properties of polynomial coefficient extraction, this alternating sum can be elegantly expressed as the coefficient of $x^{n-1}$ in a Jacobi polynomial $$Nar_\nu(-1) =  \frac{1}{n} [x^{n-1}] (1-x)^{mn} (1+x)^n .$$
\end{remark}

\begin{corollary}
    The reduced Euler characteristic of the canonical join complex of the Dyck lattice is either zero or a signed Catalan number. 
\end{corollary}
\begin{proof}
    By Proposition~\ref{theorem_reciprocity} the reduced Euleer characteristic is given by $-\text{Nar}_\nu(-1)$ for $\nu=(NE)^n$. Using Remark~\ref{remark_polynomial}, this is equal to
   $$\text{Nar}_\nu(-1)=-\frac{1}{n}[x^{n-1}](x^2-1)^n.$$
For $n$ even, the coefficient vanishes. For $n=2r+1$ it simplifies to $(-1)^r \frac{1}{2r+1} \binom{2r+1}{r}=(-1)^r C_r$, where $C_r$ is the $r$th Catalan number.
\end{proof}
%\begin{proof}
%    By Proposition~\ref{theorem_reciprocity} the reduced Euler characteristic is given by the alternating sum of classical Narayana numbers, given by $N(n,k)=\frac{1}{n} \binom{n}{k}\binom{n}{k-1}$. Their alternating sum is well-known to be either zero or a signed Catalan number.
%\end{proof}

%\subsubsection{Another way to match elements}
\begin{remark}
In Definition~\ref{definition_nara_poly}, we defined the Narayana polynomial in terms of valleys, following the convention of \cite{tamcom}. We should note that there is an elegant bijection between the faces of $\Delta_{\nu, \text{Dyck}}$ and $\nu$-paths. This is achieved by taking the unique northwest path with peaks $NW$ at the marked boxes and mirroring it, as illustrated in \Cref{Fig_bij_nara}. This approach leads similar results via peaks. %This approach yields similar results to the $\nu$-Narayana polynomial defined via peaks. But the two definitions do not coincide for arbitrary~$\nu$.
\begin{figure}[h!]
  \centering
  \begin{subfigure}[b]{0.45\textwidth}
    \centering
    \resizebox{0.6\linewidth}{!}{\includegraphics[page=1]{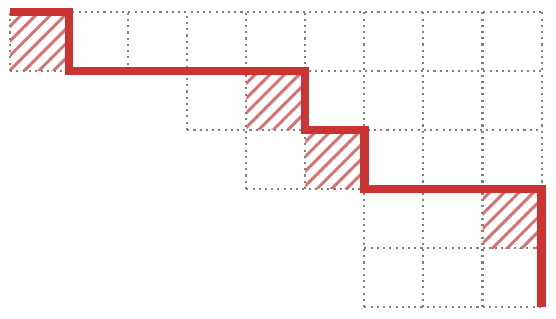}}
    \caption{A face in $\Delta_{\nu,Dyck}$ and a northwest path with peaks at the boxes.}
    \label{Fig_bij_nara1}
  \end{subfigure}\hfill
  \begin{subfigure}[b]{0.45\textwidth}
    \centering
    \resizebox{0.6\linewidth}{!}{\includegraphics[page=1]{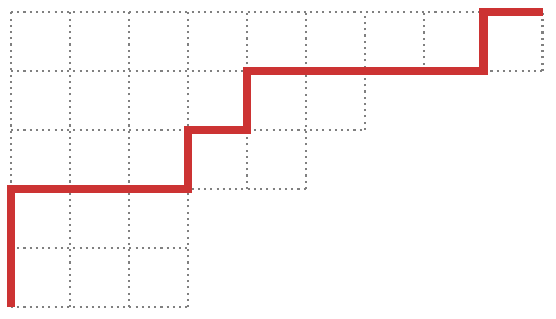}}
    \caption{The corresponding $\nu$-path for the face in \Cref{Fig_bij_nara1}.}
    \label{Fig_bij_nara2}
  \end{subfigure}
 \caption{An alternative bijection form faces of $\Delta_{\nu,Dyck}$ to $\nu$-Dyck paths.}
  \label{Fig_bij_nara}
\end{figure}
\end{remark}

\subsection{Shellability}\label{section::shellability}
Finally, we study the shellability of the canonical join complex of alt $\nu$-Tamari lattices. We start with a brief overview of basic notions following~\cite{Anders1996, Anders1997,wachs2007poset}.
%A facet $F_k$ in a sequence of facets is a \defn{homology facet} if its intersection with the preceding facets, $\left( \bigcup_{i=1}^{k-1} \langle F_i \rangle \right) \cap \langle F_k \rangle$, is equal to $\langle F_k \rangle$. 
A non-pure simplicial complex $\Delta$ is \defn{shellable} if there exists an ordering of its facets~$F_1, F_2, \dots, F_m$ (a \defn{shelling}) such that each $j>1$, $F_j$ intersects with the complex $\Delta_{j-1}$ generated by $F_1,\dots,F_{j-1}$ in a pure $(\dim F_j)-1$ complex.
Each facet $F_j$ in a shelling contains a minimal face, that has not appeared before, given by $R(F_j)=\{x \in F_j : F_j \setminus \{x\} \subseteq F_k \text{ for some } k < j\}.$ We call~$F_j$ a \defn{homology facet} if $R(F_j)=F_j$ holds.

%A facet $F_k$ is a \defn{homology facet} if its closure $\langle F_k \rangle$ is equal to~$\left( \bigcup_{i=1}^{k-1} \langle F_i \rangle \right) \cap \langle F_k \rangle$.

The notion of shellability connects directly to our previous findings. Specifically, Bj\"{o}rner and Wachs showed that every non-pure vertex decomposable complex is shellable \cite[Theorem 11.3]{Anders1997}. In particular, a shelling sequence can be constructed inductively. For a be a decomposing vertex $v$ of $\Delta$. Suppose $F_1, \dots, F_a$ is a shelling of the deletion~$\Delta \setminus \{v\}$ and $E_1, \dots, E_b$ a shelling of the link $\ell k_{\Delta}(v)$. As established in \cite[Lemma 6]{Wachs1999}% (see also \cite{Anders1996})
, the concatenated sequence 
    \[ F_1, \dots, F_a, E_1 \cup \{v\}, \dots, E_b \cup \{v\} \] 
    forms a shelling of $\Delta$. 
Thus, \Cref{vertex_decomposable} yields Corollary \ref{cor_shellable}.

%By \cite[Theorem 11.3]{Anders1997}, every non-pure vertex decomposable complex is shellable. Therefore, \Cref{vertex_decomposable} implies the following.

\begin{repcorollary}
     The canonical join complex of the $\text{alt}$ $\nu$-Tamari lattice is shellable.
\end{repcorollary}
Furthermore, as every right interval of the classical Tamari lattice is an alt~$\nu$-Tamari lattice \cite{CCh2024}, we also obtain the following.
\begin{corollary}
The canonical join complex of any right interval of the classical Tamari lattice is vertex decomposable and, consequently, shellable.
\end{corollary}

\subsubsection{Explicit shelling sequence} 
By suitably restricting the recursive construction, we can provide an explicit description of a shelling. To this end, we label the boxes according to Definition~\ref{def_labeling}.

\begin{definition}[Box-labeling]\label{def_labeling}
Let $q_1$ be the rightmost box in the bottom row of the shape $F_{\delta,\nu}$. Moving upwards, we identify $q_2, \dots, q_n$ by selecting the rightmost box in each row compatible with the chosen boxes below, if possible. This follows~Remark~\ref{remark_uparrow_boxes}. We assign the label $1$ to all boxes that lie in the same row or column as $q_1$, and proceed to label the remaining boxes in the rows and columns of $q_2, \dots, q_n$ accordingly. This is illustrated in \Cref{Fig_shelling}.
\begin{figure}[!h]
    \centering
    \includegraphics[width=0.7\textwidth]{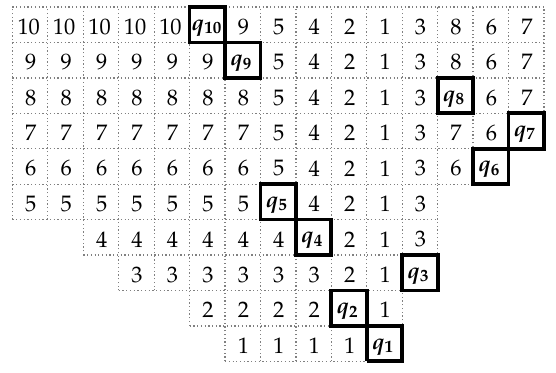} 
    \caption{Box-labeling according to Definition~\ref{def_labeling}.}
    \label{Fig_shelling}
\end{figure}    
\end{definition}
Using the box-labeling from Definition~\ref{def_labeling}, we can describe an explicit shelling of the canonical join complex. This is made precise in Definition~\ref{def_shelling_sequence}.
\begin{observation}\label{remark_comp}
    Note that each box labeled by $i\in[n]$ is compatible to the boxes labeled $q_{1},\dots,q_{i-1}$.
\end{observation}
\begin{definition}\label{def_sequences}
Each facet of the canonical join complex $\Delta_{\delta,\nu}$ corresponds to a maximal collection of pairwise compatible boxes in the shape $F_{\delta,\nu}$. With respect to the box-labeling introduced above, a facet $F$ is represented by a sequence $a(F)=(a_1,\dots,a_n)$, where, for each $i\in[n]$,
\[
a_i=
\begin{cases}
q_i & \text{if the facet contains the box labeled } q_i,\\
i   & \text{if the facet contains a box labeled } i,\\
0  & \text{if the facet contains neither a box labeled } i \text{ nor } q_i.
\end{cases}
\]
Notice that this representation is not injective. In particular, distinct facets of $\Delta_{\delta,\nu}$ may be represented by the same sequence.
\end{definition}

\begin{definition}[Shelling sequence]\label{def_shelling_sequence}
    Identify the facets $F_1, F_2, \dots, F_m$ of the canonical join complex $\Delta_{\delta,\nu}$ with their sequences from Definition~\ref{def_sequences}, with possible repetition. We order the facets according to their sequences in reverse lexicographic order, using the relations $q_i > i>0$. We denote the corresponding shelling by~$\mathcal{F}: F_1, F_2, \dots, F_m$, where the initial facet is always $F_1 = \{q_1, \dots, q_n\}$.
\end{definition}

\begin{remark}
    Note that the ordering in Definition~\ref{def_shelling_sequence} is well-defined, since no facet contains two boxes labeled with the same $i \in [n]$. Furthermore, we do not impose any order on the facets corresponding to the same sequence. 
\end{remark}

\begin{example}\label{Fig_ex_lattice2}
 We continue Example~\ref{Fig_ex_lattice} for $\nu=(2,3,0)$ and~$\delta=(3,0)$. The canonical join complex~$\Delta_{\delta,\nu}$ has facets $F_1,\dots,F_8$, illustrated in~\cref{ex_complex}.
    \begin{figure}[!h]
    \centering
    \includegraphics[width=1\textwidth]{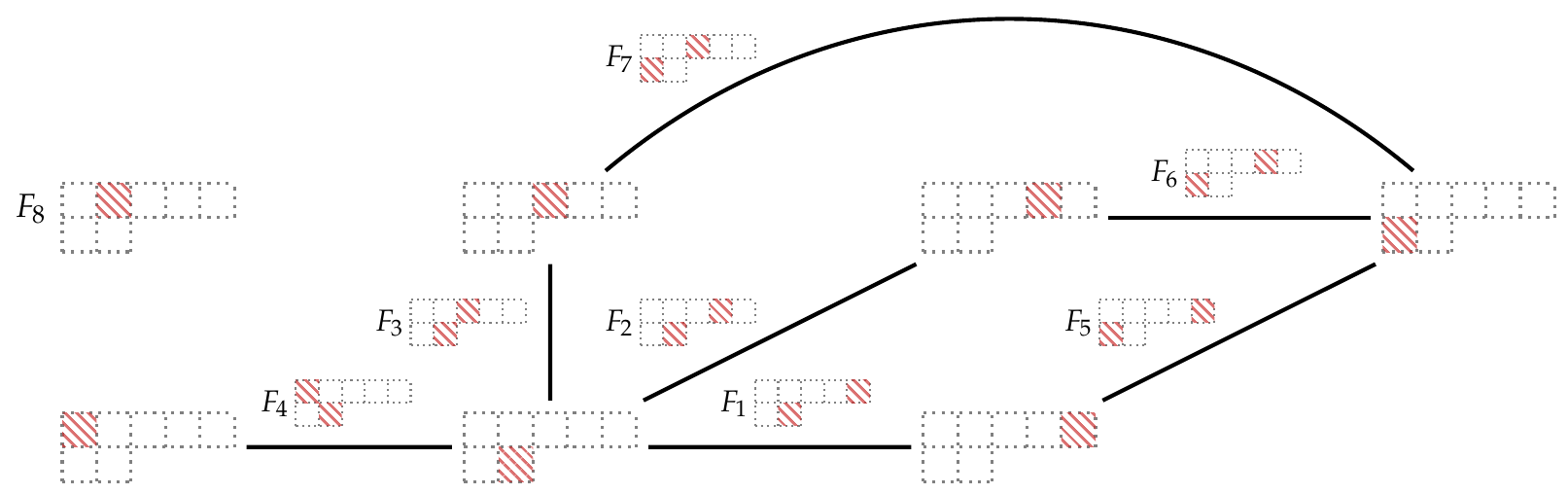} 
    \caption{The canonical join complex $\Delta_{\delta,\nu}$ for $\nu=(2,3,0)$ and~$\delta=(3,0)$. The lattice is illustrated in \cref{Fig_ex_lattice}.}
    \label{ex_complex}
\end{figure}
Following Definition~\ref{def_sequences}, we identify the facets $F_1,\dots,F_8$ with the sequences $a(F_1)=(q_1,q_2),\, a(F_2)=a(F_3)=a(F_4)=(q_1,2),$ $ a(F_5)=(1,q_2), a(F_6)= a(F_7)=(1,2)$, $ a(F_8)=(1,0)$, as illustrated in \Cref{ex_sequences}. Ordering these sequences in reverse lexicographic order yields: $$(q_1,q_2)> (q_1,2)>(1,q_2)> (1,2)> (1,0).$$
    \begin{figure}[!h]
    \centering
    \includegraphics[width=.9\textwidth]{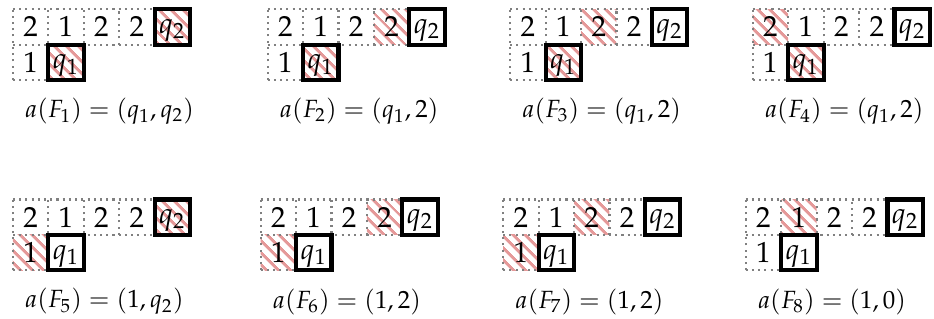} 
    \caption{Sequences corresponding to the facets of \cref{ex_complex}.}
    \label{ex_sequences}
\end{figure}

One can verify that any sequence preserving this order forms a shelling of $\Delta_{\delta,\nu}$. These are listed explicitly below:
\begin{align*}
\mathcal{F}_{1}  &: F_1, F_2, F_3, F_4, F_5, F_6,F_7,F_8 & \mathcal{F}_{7}  &: F_1, F_2, F_3, F_4, F_5, F_7,F_6,F_8 \\
\mathcal{F}_{2}  &: F_1, F_2, F_4, F_3, F_5, F_6,F_7,F_8 & \mathcal{F}_{8}  &: F_1, F_2, F_4, F_3, F_5, F_7,F_6,F_8 \\
\mathcal{F}_{3}  &: F_1, F_3, F_2, F_4, F_5, F_6,F_7,F_8 & \mathcal{F}_{9}  &: F_1, F_3, F_2, F_4, F_5, F_7,F_6,F_8 \\
\mathcal{F}_{4}  &: F_1, F_3, F_4, F_2, F_5, F_6,F_7,F_8 & \mathcal{F}_{10} &: F_1, F_3, F_4, F_2, F_5, F_7,F_6,F_8 \\
\mathcal{F}_{5}  &: F_1, F_4, F_2, F_3, F_5, F_6,F_7,F_8 & \mathcal{F}_{11} &: F_1, F_4, F_2, F_3, F_5, F_7,F_6,F_8 \\
\mathcal{F}_{6}  &: F_1, F_4, F_3, F_2, F_5, F_6,F_7,F_8 & \mathcal{F}_{12} &: F_1, F_4, F_3, F_2, F_5, F_7,F_6,F_8.
\end{align*}
\end{example}

\begin{proposition}\label{prop_shelling}
    Each ordering $\mathcal{F}$, introduced in Definition~\ref{def_shelling_sequence}, is a shelling for the canonical join complex of the alt~$\nu$-Tamari lattice.
\end{proposition}

\begin{proof}
Let $\mathcal{F}: F_1, F_2, \dots, F_m$ be an ordering of the facets of the canonical join complex $\Delta_{\delta,\nu}$ as in Definition~\ref{def_shelling_sequence}, and identify each facet $F_j$ with the sequence $a_j$. For~$j > 1$, let $\Delta_{j-1}$ denote the complex generated by $F_1,\dots,F_{j-1}$. To show that $F_j$ intersects $\Delta_{j-1}$ in a pure $(\dim(F_j)-1)$-dimensional complex, it suffices to show that each facet $F_j$ contains a unique vertex $v \in F_j$ such that the face $F_j \setminus \{v\}$ has already appeared in a previous facet of the ordering. 

Since $F_j$ is a facet and $a_{j-1}>a_j$, the sequence $a_j$ contains at least one index $i$ such that the $i$th entry of $a_j$ is equal to $i$. Let $v$ be the vertex of $F_j$ labeled by the smallest label $i\in\{1,\dots,n\}$ appearing in $F_j$. Because $i$ is minimal, the first~$i-1$ entries of $a_j$ are in $\{0,q_1,\dots,q_{i-1}\}$ and by construction $q_i$ is compatible with~$q_1,\dots,q_{i-1}$. By Observation~\ref{remark_comp}, $q_i$ is compatible with all other vertices of $F_j$, corresponding to labels $\geq i+1$. Hence~$F_j\setminus\{v\} \cup\{q_i\}$ is a face of the complex. Now, every facet containing the face $F_j\setminus\{v\} \cup\{q_i\}$ is listed before~$F_j$, since its sequence is of the form $(a_{j,1},\dots a_{j,i-1},q_i,\dots)$, appearing before~$a_j$ in reverse lexicographic order.
\end{proof}

%\subsubsection{Proof of~\Cref{theorem_wedge}}
%\subsubsection{Enumeration of homology facets in general}
\subsubsection{An invariance property}
Our next purpose is to prove that, for a fixed path\footnote{We associate the northeast path $NE^{\nu_1}\dots NE^{\nu_n}$ with the sequence $\nu=(\nu_1,\dots,\nu_n)$.}~$\nu=NE^{\nu_1}\dots NE^{\nu_n}$, the canonical join complexes of all alt $\nu$-Tamari lattices have the same number of homology facets of size $(n-1)$.

\begin{convention}\label{remark_shelling}
We use the convention of listing facets in the shelling corresponding to the same sequence according to their boxes from \textbf{right to left}.
\end{convention}

\begin{definition}\label{def_tuple}
Let $\nu = (\nu_1, \dots, \nu_n)$, and let $F$ be a facet of the canonical join complex $\Delta_{\delta,\nu}$. The facet $F$ corresponds to a set of pairwise compatible boxes (shaded boxes) of shape $F_{\delta,\nu}$ consisting of $n-1$ rows. We consider row by row from bottom to top. Within each row, we enumerate the boxes from right to left that are compatible with the shaded boxes in the rows below. We identify $F$ with the tuple~$a=(a_1, \dots, a_{n-1})$, where~$a_i$ denotes the label of the marked box in row $i$. This construction is illustrated in~\Cref{Fig_sequencehf}.

    \begin{figure}[!h]
    \centering
    \includegraphics[width=0.5\textwidth]{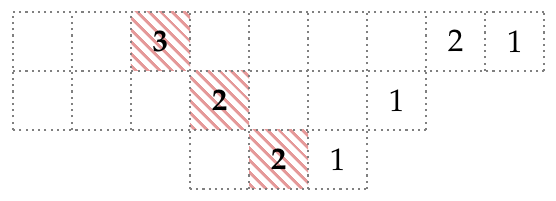} 
    \caption{Facet for $a=(2,2,3)$ in the shape $F_{\delta,\nu}$ for $\nu=(3,4,2,1)$ and $\delta=(1,2,0)$.}
    \label{Fig_sequencehf}
\end{figure}
\end{definition}

The following map is the key for the proof of \Cref{theorem_wedge}.

\begin{definition}
Let $\nu = (\nu_1, \dots, \nu_n)$ and let $\delta, \delta'$ be increment vectors with respect to $\nu$. Let 
$$\varphi : \operatorname{Facets}_{n-1}(\Delta_{\delta,\nu}) \longrightarrow \operatorname{Facets}_{n-1}(\Delta_{\delta',\nu})$$
be the map sending each $(n-1)$-facet $F \in \operatorname{Facets}_{n-1}(\Delta_{\delta,\nu})$ to the unique $(n-1)$-facet of $\Delta_{\delta',\nu}$ associated with the same tuple under Definition~\ref{def_tuple}.
\end{definition}

\begin{remark}
    While the set $\operatorname{Facets}_{n-1}(\Delta_{\delta,\nu})$ may be empty in general, it is guaranteed to be non-empty whenever $\nu_i \ge 1$ for all $i$.
\end{remark}

\begin{lemma}\label{lemma_bijection_labels}
    The map $\varphi$ is well-defined and a bijection.
\end{lemma}
\begin{proof}
If $\operatorname{Facets}_{n-1}(\Delta_{\delta,\nu})$ is empty, it is impossible to mark a box in each row of the shape $F_{\delta,\nu}$ in a compatible way. This implies that $\operatorname{Facets}_{n-1}(\Delta_{\delta',\nu})$ is also empty, and thus the statement holds trivially.

Let $F \in \operatorname{Facets}_{n-1}(\Delta_{\delta,\nu})$ be a facet of size $(n-1)$. Then $F$ contains one box in each row.
We proceed by induction on $n$, the length of $\nu$. The base case $n=1$ is trivially true.
Now, assume we have chosen the $n-2$ bottom boxes $b'_1,\dots,b'_{n-2}$ according to the definition of $\varphi$. 
Consider the largest rectangles inside the shape~$F_{\delta',\nu}$ that have the boxes~$b'_1,\dots,{b'_{n-2}}$ as their bottom-left corners. By the definition of compatibility, no box can be marked in these rectangles.
By construction, the top row contains the same number of compatible positions in both shapes. Hence, there is a unique position for the marked box in the top row, which establishes the claim.
\end{proof}

\begin{example}
An example of the bijection $\varphi$ for $\nu=(3,4,2,1)$ and $\delta_{left}=(0,0,0)$, $\delta_{middle}=(1,2,0)$, $\delta_{right}=(4,2,0)$ is illustrated in~\Cref{Fig_varphi}.

\begin{figure}[!h]
  \centering
  \begin{subfigure}[b]{0.3\textwidth}
    \centering
    \resizebox{1\linewidth}{!}{\includegraphics[page=1]{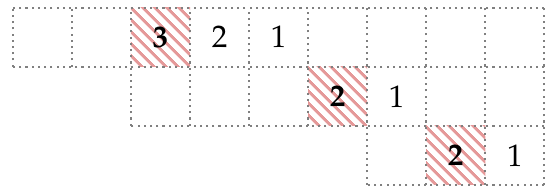}}
    \caption{$\delta_{left}=(0,0,0)$}
    \label{Fig_varphi1}
  \end{subfigure}\hfill
  \begin{subfigure}[b]{0.3\textwidth}
    \centering
    \resizebox{1\linewidth}{!}{\includegraphics[page=1]{figures/Fig_varphi2.pdf}}
    \caption{$\delta_{middle}=(1,2,0)$}
    \label{Fig_varphi2}
  \end{subfigure}\hfill
    \begin{subfigure}[b]{0.3\textwidth}
    \centering
    \resizebox{1\linewidth}{!}{\includegraphics[page=1]{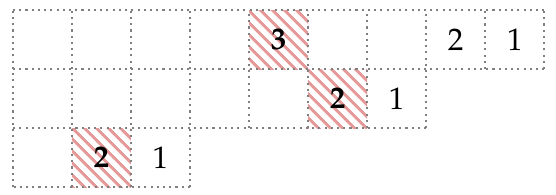}}
    \caption{$\delta_{right}=(4,2,0)$}
    \label{Fig_varphi3}
  \end{subfigure}\hfill
  \caption{The bijection $\varphi$ between $(n-1)$-facets of $\Delta_{\delta,\nu}$ for fixed~$\nu=(3,4,2,1)$ and varying $\delta$.}
  \label{Fig_varphi}
\end{figure}
\end{example}

\begin{theorem}\label{thm_hf_invariance}
Let $\nu$ be a finite northeast path. Then the number of homology facets of size $(n-1)$ for the canonical join complex $\Delta_{\delta,\nu}$ is invariant under the choice of the increment vector $\delta$.
\end{theorem}
\begin{proof}
By Lemma~\ref{lemma_bijection_labels}, the mapping $\varphi$ establishes a bijection between the $(n-1)$-facets. It remains to be shown that $\varphi$ restricts to a bijection between $(n-1)$-homology facets to $(n-1)$-homology facets. Let $F_1, \dots, F_m$ denote a shelling of~$\Delta_{\delta,\nu}$ as introduced in Remark~\ref{remark_shelling}, ordered from right to left for boxes sharing the same label.

Let $F\in \operatorname{Facets}_{n-1}(\Delta_{\delta,\nu})$ be a homology facet of size $n-1$. Then $F$ contains a box in each row. Moreover, if $F$ is a homology facet then each box $b$ in it can be replaced by a box to its right in the same row, forming a new facet $F'$ (because there is a facet $F'$ containing $F\setminus \{b\}$ before in the shelling). The image $\varphi(F) \in \operatorname{Facets}_{n-1}(\Delta_{\delta',\nu})$ also satisfies this property by construction of $\varphi$. Therefore, $\varphi$ is a homology facet.
\end{proof}

\begin{remark}
    According to \cite[Theorems 3.4 and 4.1]{Anders1996}, a shellable complex is homotopy equivalent to a wedge of spheres, where each $r$-sphere corresponds to an $r$-dimensional homology facet. Therefore, by Theorem~\ref{thm_hf_invariance} the number of~$(n-2)$-dimensional spheres for the canonical join complex $\Delta_{\delta,\nu}$ is invariant under~$\delta$.
%Consequently, to prove \Cref{theorem_wedge}, it suffices to establish Theorem~\ref{thm_hf_invariance}.
\end{remark}

\begin{remark}
    The number of homology facets of size $<(n-2)$ is not invariant under change of the increment vector $\delta$. Examples are listed in~\cref{tab:homology-dyck-uniform}.
\end{remark}

\subsubsection{Counting top-dimensional spheres}
Our next objective is to enumerate the homology facets of size $(n-1)$. For our counting formulation, we consider a finite northeast path $\nu=NE^{\nu_1} \dots NE^{\nu_n}$ with $\nu_i \geq 2$. %The key to this enumeration is given in Definition~\ref{def_shrunked}.

%\begin{figure}[!h]
%    \centering
%    \includegraphics[width=0.4\textwidth]{figures/Fig_boxes_specialcase.pdf} 
%    \caption{Boxes contained in a facet of size $(n-1)$.}
%    \label{Fig_boxes_specialcase}
%\end{figure}

\begin{definition}\label{def_shrunked}
    Let $\nu = ({\nu_1},\dots ,{\nu_n})$ be a finite northeast path with $\nu_i \geq 2$. We define the \defn{shrunken path}, denoted by $\bar{\nu}$, as the path~$\bar{\nu} = ({\nu_1-2},\dots, {\nu_{n}-2})$.
\end{definition}

Since by Theorem~\ref{thm_hf_invariance} the number of homology facets of size $(n-1)$ is invariant under the choice of $\delta$, we restrict our attention to the case $\delta^{\text{min}}=(0,\dots,0)$. This setting admits an elegant bijection, introduced in Definition~\ref{def_H_bijection}.

\begin{definition}\label{def_H_bijection}
Let $\nu = (\nu_1, \dots, \nu_n)$ be a finite northeast path with $\nu_i \geq 2$ for all~$i$. Within the shape $F_{\bar{\nu}}$ associated with the shrunken path $\bar{\nu}$, we label the~$j$th vertical segment of the $i$th row as $r_{i,j}$, where rows are indexed from bottom to top and vertical segments from left to right, see \Cref{Fig_hf2}.
We assign labels to the boxes of $F_{\delta^{\text{min}},\nu}$ as follows: in each row $i$, the~$j$th box from the right - omitting the first~$1+2(i-1)$ rightmost boxes - is labeled with $r_{i,j}$, see \Cref{Fig_hf1}. 
\begin{figure}[h!]
  \centering
\begin{subfigure}[b]{0.45\textwidth}
    \centering
    \resizebox{0.5\linewidth}{!}{\includegraphics[page=1]{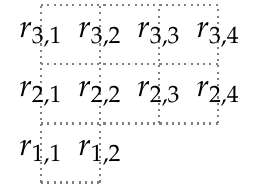}}
    \caption{Labeling of vertical line segments $F_{\bar{\nu}}$}
    \label{Fig_hf2}
  \end{subfigure}
\hfill
    \begin{subfigure}[b]{0.5\textwidth}
    \centering
    \resizebox{1\linewidth}{!}{\includegraphics[page=1]{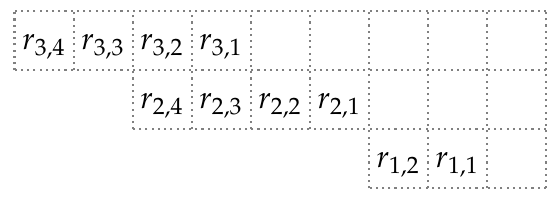}}
    \caption{Box labeling in $F_{\delta^{\text{min}},\nu}$}
    \label{Fig_hf1}
  \end{subfigure}
  
 \caption{Bijection between shrunken Dyck paths and homology facets of size $(n-1)$. %of~$\Delta_{\delta^{\text{min}}, \nu}$
 }
  \label{Fig_shrunken}
\end{figure}
\end{definition}

\begin{definition}
Let $\nu=(\nu_1,\dots,\nu_n)$ be a finite northeast path with $\nu_i \geq 2$, and let~$\bar{\nu}$ be its shrunken path. We define the map
\[
\mathcal{H} : \{\text{$\bar{\nu}$-Dyck paths}\} \longrightarrow \{\text{homology facets of $\Delta_{\delta^{min},\nu}$ with size } (n-1)\},
\]
which maps a~$\bar{\nu}$-Dyck path with vertical line segments $\{ r_{i_1,j_1}, \dots, r_{i_{n-1},j_{n-1}} \}$ to the collection of boxes sharing the same labels.
\end{definition}

\begin{lemma}\label{lemma_facet}
    Let $\nu=(\nu_1,\dots,\nu_n)$ be a finite northeast path, with $\nu_i \geq 2$ and $\mu$ a~$\bar{\nu}$-Dyck path. Then~$\mathcal{H}(\mu)$ is a facet of the canonical join complex for the alt $\nu$-Tamari lattice~$\altTam{\nu}{\delta}$ of size $(n-1)$.
\end{lemma}
\begin{proof}
    Let $\mu$ be a $\bar{\nu}$-Dyck path. Suppose, for contradiction, that $\mathcal{H}(\mu)$ contains two pairwise incompatible boxes labeled~$r_{i_1,j_1}$ and~$r_{i_2,j_2}$ with $i_1<i_2$. 
    We consider the number of boxes in $F_{\delta^{\text{min}}, \nu}$ labeled by an $r_{i,j}$ to the right of $r_{i_1,j_1}$ and~$r_{i_2,j_2}$. By construction, we omitted in row $i_2$ exactly $2(i_2-i_1)>2$ boxes more than in row $i_1$. Since~$r_{i_2,j_2}$ is incompatible with $r_{i_1,j_1}$, the box $r_{i_2,j_2}$ must be vertical or northeast to~$r_{i_1,j_1}$. Hence there are at least $2$ more labeled boxes to the right of $r_{i_1,j_1}$ than there are for $r_{i_2,j_2}$.
    Therefore, the path $\mu$ uses in row $i_1$ a vertical line segment to the right of the used vertical line segment in row $i_2$. Since $i_1<i_2$, this contradicts the definition of a $\bar{\nu}$-Dyck path. Since $\mathcal{H}(\mu)$ contains a box in each row, it is a facet of size $n-1$.
    %Since each~$\nu_i\geq2$, there are at least~$\nu_{i_2}$ more boxes to the left of $r_{i_2,j_2}$ than to the left of $r_{i_1,j_1}$. 
\end{proof}

    \begin{lemma}\label{lemma_H_wd}
    Let $\nu=(\nu_1,\dots,\nu_n)$ be a finite northeast path, with $\nu_i \geq 2$ and $\mu$ a $\bar{\nu}$-Dyck path. Then~$\mathcal{H}(\mu)$ is a homology facet of size $(n-1)$. In particular, $\mathcal{H}$ is well-defined.
\end{lemma}

\begin{proof}
    By Lemma~\ref{lemma_facet}, it remains to show that $\mathcal{H}(\mu)$ is a homology facet. Let $F_1,\dots,F_m$ be a shelling as in Definition~\ref{def_shelling_sequence} and let us choose boxes with the same label from right to left in each row, see Convention~\ref{remark_shelling}. Without loss of generality let~$F_j=\mathcal{H}(\mu)$.
    %Since each facet appears in the shelling, let $j\in\{1,\dots,m\}$ such that $F_j=\mathcal{H}(\mu)$. 
    We show: $$F_j=\{x \in F_j : F_j \setminus \{x\} \subseteq F_k \text{ for some } k < j\}.$$
We observe that each facet of the form $\mathcal{H}(\mu)$ contains exactly one box per row. Since we omit the~$1+2(i-1)$ rightmost boxes in row~$i$ any two boxes in consecutive rows are separated horizontally by at least one column of boxes. Hence, for each vertex $x \in F_j = \mathcal{H}(\mu)$, the box $b(x)$ to the right of~$x$ is compatible with~$F_j \setminus \{x\}$. 
Consequently,~$F_j \setminus \{x\} \cup \{b(x)\}$ is a facet in the shelling $\mathcal{F}$ that appears before~$F_j$. Therefore,~$\mathcal{H}(\mu)$ is a homology facet of size~$(n-1)$.

\end{proof}

\begin{lemma}\label{lemma_H_facets}
    A facet of the canonical join complex $\Delta_{\delta^{\text{min}},\nu}$ of size $(n-1)$ appears as a homology facet, if its vertices correspond to boxes labeled by $r_{1,j_1}, \dots, r_{{n-1},j_{n-1}}$ such that~$j_{k-1} \leq j_k $ for $k\in \{2,\dots,n-1\}$.
\end{lemma}

\begin{proof}
Let $F$ be a homology facet of size $(n-1)$. We separate the shape $F_{\delta^{\text{min}},\nu}$ according to the boxes labeled by some $r_{i,j}$ by drawing a red line, as illustrated in~\cref{Fig_redline1}. 
\begin{figure}[h!]
  \centering
\begin{subfigure}[b]{0.46\textwidth}
    \centering
    \resizebox{1\linewidth}{!}{\includegraphics[page=1]{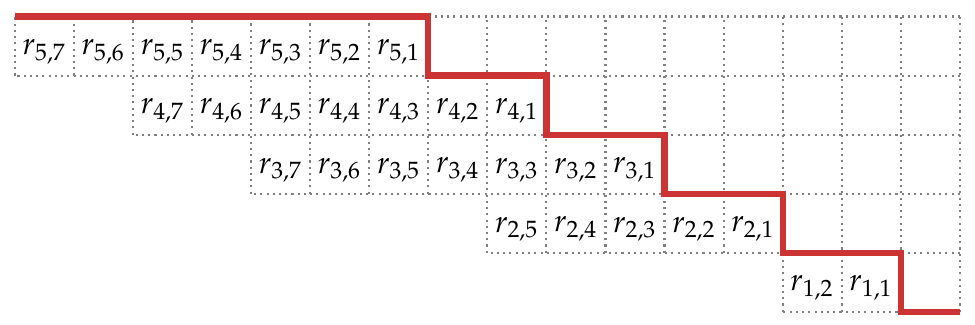}}
    \caption{Boxes below the red line are labeled by some $r_{i,j}$ by Definition~\ref{def_H_bijection}.}
    \label{Fig_redline1}
  \end{subfigure}
\hfill
    \begin{subfigure}[b]{0.46\textwidth}
    \centering
    \resizebox{1\linewidth}{!}{\includegraphics[page=1]{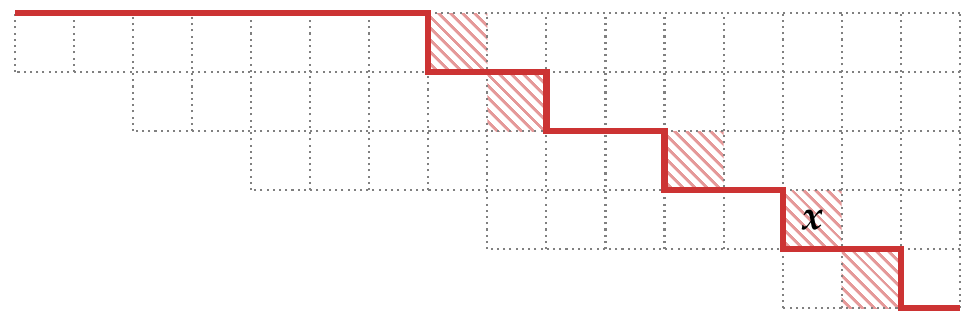}}
    \caption{A facet of size $(n-1)$ containing boxes above red line.}
    \label{Fig_redline2}
  \end{subfigure}
  
 \caption{Red line in $F_{\delta^{\text{min}},\nu}$ and a facet of $\Delta_{\delta^{\text{min}},\nu}$.
 }
  \label{Fig_redline}
\end{figure}
We establish the following:
\begin{enumerate}
    \item Each vertex of $F$ corresponds to a box labeled by $r_{i,j}$ for some $i,j$. \label{c1}
    \item The boxes corresponding to the vertices of $F$ are labeled $r_{1,j_1}, \dots, r_{n-1,j_{n-1}}$ such that $j_{k-1} \leq j_k$ for all $k \in \{2, \dots, n-1\}$.\label{c2}
\end{enumerate}

\eqref{c1}: Suppose for contradiction that $F$ contains a box above the red line. This is illustrated in \cref{Fig_redline2}. Let $x$ be the bottom most box above the red line. Since~$F$ corresponds to a set of $n-1$ boxes, one box per row, and all boxes below $x$ are below the red line, the box in the row below $x$ must be the box directly southeast to~$x$. Otherwise, it is not possible to mark $n-1$ compatible boxes. Therefore, the face~$F\setminus \{x\}$ has not appeared in a facet before. Hence $F$ cannot be a homology facet.
\eqref{c2}: Suppose for contradiction, that $j_{k-1}>j_k$ for some $k \in \{2, \dots, n-1\}$. Since all boxes of the facet $F$ are pairwise compatible, by construction, the two boxes $r_{k-1,j_{k-1}}$ and $r_{k,j_k}$ must be directly northwest of southeast to each other. But this means that $F\setminus \{r_{k,j_k}\}$ has not appeared in a facet before.

\end{proof}

\begin{proposition}\label{prop_counting_dyck}
    For a finite northeast path $\nu=(\nu_1,\dots,\nu_n)$, with $\nu_i\geq 2$, the homology facets of size $(n-1)$ of the canonical join complex $\Delta_{\delta^{\text{min}}, \nu}$ are in bijection with $\bar{\nu}$-Dyck paths.
\end{proposition}

\begin{proof}
By Lemma \ref{lemma_H_wd}, $\mathcal{H}$ is well-defined and clearly injective. Since Lemma~\ref{lemma_H_facets} gives exactly the condition for being a $\bar{\nu}$-Dyck path, the map $\mathcal{H}$ is a bijection.
\end{proof}

\begin{proof}[Proof of \Cref{theorem_wedge}]
This is a direct consequence of Theorem~\ref{thm_hf_invariance} and Propositions~\ref{prop_counting_dyck}.
\end{proof}

\begin{example}
     We continue Example~\ref{Fig_ex_lattice2} and consider $\nu=NE^2NE^3NE^2$. The shrunken Dyck path is $\bar{\nu}=(0,1,0)$. The number of $\bar{\nu}$-Dyck paths is $2$. By \Cref{theorem_wedge}, the canonical join complex of all alt $\nu$-Tamari lattices is homotopy equivalent to a wedge of spheres with top dimension~$(n-2)=1$ and the number of~$1$-spheres is $2$. We can see them in~\cref{ex_complex}.
\end{example}

\begin{corollary}\label{corollary_wedge}
    For $\nu=(NE^m)^{n}$, the canonical join complexes of all alt $\nu$-Tamari lattices $\altTam{\nu}{\delta}$ is homotopy equivalent to a wedge of finitely many spheres. Moreover, the number of~$n$~-~spheres is given by the Fuss-Catalan number$$ \frac{1}{(m-2)n+1} \binom{(m-1)n}{n}.$$
    
\end{corollary}

\begin{example}
By \cref{theorem_wedge}, all alt $\nu$-Tamari lattices have the same number of top-dimensional spheres, with the $\nu$-Tamari and $\nu$-Dyck lattices representing the two extreme cases. For $\nu=(NE^m)^n$, several examples are listed in \cref{tab:homology-dyck-uniform}, where the number of top-dimensional spheres corresponds to the Betti number~$\beta_{n-2}$. By~Corollary~\ref{corollary_wedge}, this value is given by a shifted Fuss-Catalan number.
\end{example}

\subsection{Homology}\label{section::homology}
Finally, we examine the topological structure of
the canonical join complexes of alt $\nu$-Tamari lattices.

\begin{remark}
    Our work so far has been purely combinatorial. However, \cref{realization} shows that the box complex is a geometric realization of the canonical join complex. Because simplicial complexes act as models for topological spaces, we can use our box complex as a tool to directly study these topological spaces. 
\end{remark}

Recall that the \defn{geometric realization} $|\Delta|$ of an abstract simplicial complex $\Delta$ is formed by gluing its geometric simplices along shared faces. Two such complexes,~$\Delta$ and $\Delta'$, are considered \defn{homotopy equivalent} (or of the same homotopy type) if their geometric realizations are homotopy equivalent, meaning there exist continuous maps $F: |\Delta| \to |\Delta'|$ and $G: |\Delta'| \to |\Delta|$ whose compositions are homotopic to their respective identity maps.

\begin{example}\label{example_homotopy}
The canonical join complex $\Delta_{\Tam{4}}$ of the Tamari lattice $\Tam{4}$ is given by the box complex $\Delta_{\Tam{4}}=\Delta_{(3,2,1)}$, illustrated in~\cref{Fig_cjc_tam}. The complex comprises a~$2$-simplex with three incident edges, as this structure is contractible, it is homotopy equivalent to a point. 

On the other hand, the canonical join complex $\Delta_{\Dyck{4}}$ of the Dyck lattice $\Dyck{4}$ is given by $\Delta_{\Dyck{4}}=\Delta_{(1,2,3)}$, illustrated in \cref{Fig_cjc_dyck}. It consists of two connected components: an isolated vertex and a cycle containing a filled triangle.
Consequently, $\Delta_{\Dyck{4}}$ is homotopy equivalent to the wedge of a $0$- and $1$-dimensional sphere $S^0 \vee S^1$. 
Given these distinct homotopy types, the two complexes are not homotopy equivalent.
\begin{figure}[!h]
  \centering
  \begin{subfigure}[b]{0.49\textwidth}
    \centering
    \resizebox{1\linewidth}{!}{\includegraphics[page=1]{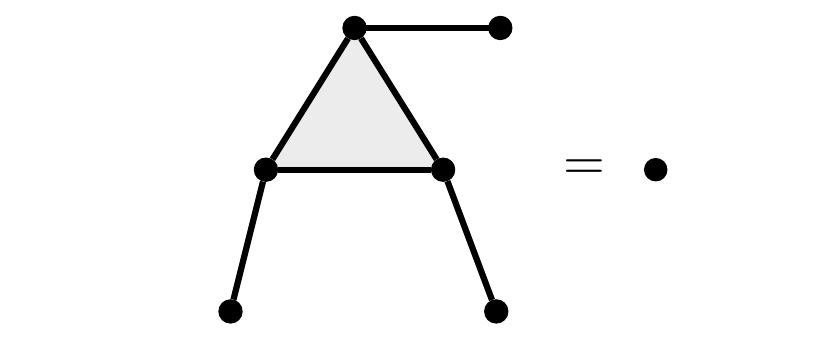}}
    \caption{Canonical join complex of $\Tam{4}$ is homotopically equivalent to a point}
    \label{Fig_cjc_tam}
  \end{subfigure}\hfill
  \begin{subfigure}[b]{0.49\textwidth}
    \centering
    \resizebox{1\linewidth}{!}{\includegraphics[page=1]{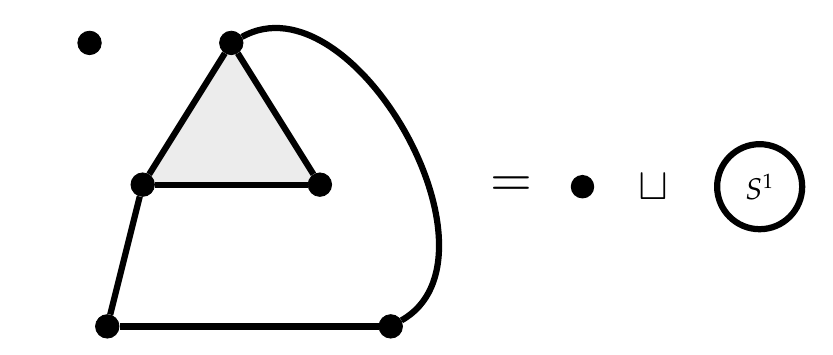}}
    \caption{Canonical join complex of $\Dyck{4}$ is homotopically equivalent to $S^0\vee S^1$.}
    \label{Fig_cjc_dyck}
  \end{subfigure}\hfill
  \caption{The canonical join complex of the Tamari $\Tam{4}$ and Dyck lattice $\Dyck{4}$ are not homotopy equivalent.}
  \label{Fig_complexes}
\end{figure}
\end{example}

\begin{remark}
 Example~\ref{example_homotopy} demonstrates that, for a fixed $\nu$, the canonical join complexes of all alt~$\nu$-Tamari lattices are not homotopy equivalent in general.
\end{remark}

As a direct consequence of shellability, the canonical join complex of a general alt $\nu$-Tamari lattice is homotopy equivalent to a wedge of spheres. Consequently, its homotopy type is completely determined by its Betti numbers. Specifically, for~$k > 0$, the $k$th Betti number $\beta_k$ denotes the number of $k$-dimensional spheres in the wedge sum, while $\beta_0$ indicates the number of connected components.

\begin{example}
    In the specific cases of $m$-Tamari and $m$-Dyck lattices, the highest-dimensional spheres are enumerated by Fuss-Catalan numbers, see Corollary~\ref{corollary_wedge}. \cref{tab:homology-dyck-uniform} summarizes the Betti numbers for the canonical join complexes associated with several~$m$-Tamari and $m$-Dyck lattices. The Betti number $\beta_{n-2}$ corresponds to the number of top-dimensional spheres and depends only on the choice of $m$ and~$n$ in \cref{tab:homology-dyck-uniform}.
\end{example}

\begin{table}[htbp]
\centering
\renewcommand{\arraystretch}{1.15}
\setlength{\tabcolsep}{9pt}
\begin{tabular}{@{}ccp{5.2cm}p{6.8cm}@{}}
\toprule
\(m\) & \(n\) & $m$-Tamari lattice $\altTam{n}{m}$ & $m$-Dyck lattice $\altDyck{n}{m}$ \\
\midrule
2 & 2 & \(\boldsymbol{\beta}=(2)\) & \(\boldsymbol{\beta}=(2)\) \\
2 & 3 & \(\boldsymbol{\beta}=(2,1)\) & \(\boldsymbol{\beta}=(2,1)\) \\
2 & 4 & \(\boldsymbol{\beta}=(1,4,1)\) & \(\boldsymbol{\beta}=(2,5,1)\) \\
2 & 5 & \(\boldsymbol{\beta}=(1,2,10,1)\) & \(\boldsymbol{\beta}=(2,8,15,1)\) \\
2 & 6 & \(\boldsymbol{\beta}=(1,0,15,20,1)\) & \(\boldsymbol{\beta}=(2,11,40,35,1)\) \\
\midrule
3 & 2 & \(\boldsymbol{\beta}=(3)\) & \(\boldsymbol{\beta}=(3)\) \\
3 & 3 & \(\boldsymbol{\beta}=(2,5)\) & \(\boldsymbol{\beta}=(2,5)\) \\
3 & 4 & \(\boldsymbol{\beta}=(1,8,14)\) & \(\boldsymbol{\beta}=(2,9,14)\) \\
3 & 5 & \(\boldsymbol{\beta}=(1,2,45,42)\) & \(\boldsymbol{\beta}=(2,13,55,42)\) \\
\midrule
4 & 2 & \(\boldsymbol{\beta}=(4)\) & \(\boldsymbol{\beta}=(4)\) \\
4 & 3 & \(\boldsymbol{\beta}=(2,12)\) & \(\boldsymbol{\beta}=(2,12)\) \\
4 & 4 & \(\boldsymbol{\beta}=(1,12,55)\) & \(\boldsymbol{\beta}=(2,13,55)\) \\
\midrule
5 & 2 & \(\boldsymbol{\beta}=(5)\) & \(\boldsymbol{\beta}=(5)\) \\
5 & 3 & \(\boldsymbol{\beta}=(2,22)\) & \(\boldsymbol{\beta}=(2,22)\) \\
5 & 4 & \(\boldsymbol{\beta}=(1,16,140)\) & \(\boldsymbol{\beta}=(2,17,140)\) \\
\midrule
6 & 2 & \(\boldsymbol{\beta}=(6)\) & \(\boldsymbol{\beta}=(6)\) \\
6 & 3 & \(\boldsymbol{\beta}=(2,35)\) & \(\boldsymbol{\beta}=(2,35)\) \\
6 & 4 & \(\boldsymbol{\beta}=(1,20,285)\) & \(\boldsymbol{\beta}=(2,21,285)\) \\
\midrule
7 & 2 & \(\boldsymbol{\beta}=(7)\) & \(\boldsymbol{\beta}=(7)\) \\
7 & 3 & \(\boldsymbol{\beta}=(2,51)\) & \(\boldsymbol{\beta}=(2,51)\) \\
\bottomrule
\end{tabular}
\caption{Betti numbers $\boldsymbol{\beta}=(\beta_0,\beta_1,\dots,\beta_{n-2}) $ of the canonical join complex of $m$-Tamari and $m$-Dyck lattices. %$\altTam{\nu}{\delta}$ where $\nu=(NE^m)^n$, considering two extreme cases.
}
\label{tab:homology-dyck-uniform}
\end{table}

\subsection{Cross Tamari Lattices}

The definition of the alt $\nu$-Tamari lattice extends naturally to cross shapes, where row permutations of the shape are also permitted. These lattices are formally introduced in \cite[Section 2.4]{vonbell2025framing} as a special case of framing lattices. Moreover, Ceballos and von Bell present a definition using the same structure as for $(\delta,\nu)$-trees in \cref{section::altvtamari}.

We remark that in the literature, cross Tamari lattices were first studied under the name of chute move lattices \cite{jonsson2005, rubey2012maximal,axelrodfreed2025chute,billey2025lattice}. These lattices depend on a parameter $r$, with the cross Tamari lattices corresponding precisely to the case $r=1$.

\begin{example}
Consider the cross Tamari lattice for a cross shape formed by three vertical and three horizontal boxes sharing a central box. \cref{Fig_cross_tamari} illustrates the cross Tamari lattice, while \cref{Fig_cross_lattice} depicts its perspective edge-labeling. Its canonical join complex comprises five vertices and four edges, \cref{Fig_cross_cjc}. We observe that this canonical join complex is isomorphic to the corresponding box complex, illustrated in \cref{Fig_cross_boxcomplex}. This serves as the primary motivation for Conjecture~\ref{conjectutre_cross}.

\begin{figure}[!h]
  \centering
  \begin{subfigure}[b]{0.3\textwidth}
    \centering
    \resizebox{1\linewidth}{!}{\includegraphics[page=1]{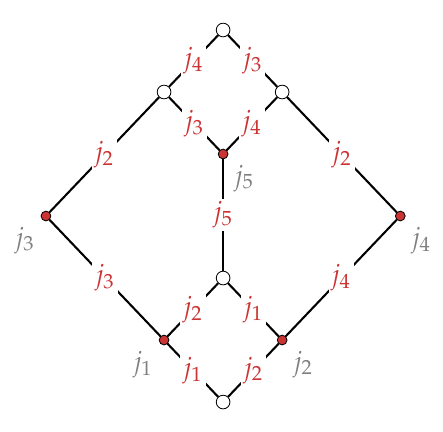}}
    \caption{The perspective edge-labeling for \cref{Fig_cross_tamari}.}
    \label{Fig_cross_lattice}
  \end{subfigure}\hfill
  \begin{subfigure}[b]{0.3\textwidth}
    \centering
    \resizebox{1\linewidth}{!}{\includegraphics[page=1]{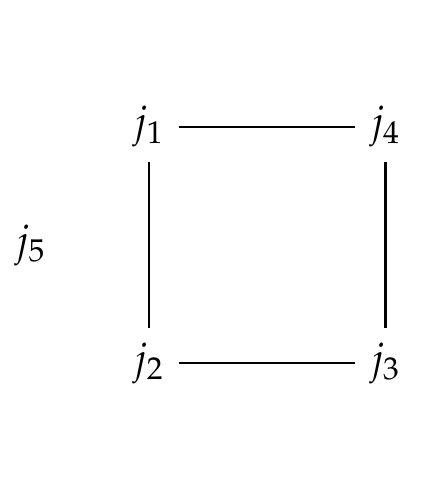}}
    \caption{The canonical join complex for \cref{Fig_cross_tamari}.}
    \label{Fig_cross_cjc}
  \end{subfigure}\hfill
\begin{subfigure}[b]{0.3\textwidth}
    \centering
    \resizebox{1\linewidth}{!}{\includegraphics[page=1]{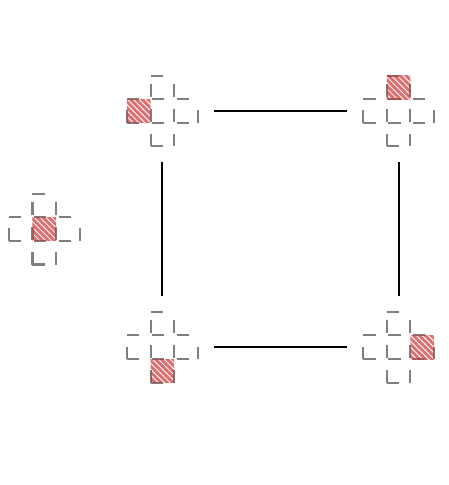}}
    \caption{The box complex for \Cref{Fig_cross_tamari}.}
    \label{Fig_cross_boxcomplex}
  \end{subfigure}
  \caption{The perspective edge-labeling, canonical join complex and corresponding box complex for the lattice in \cref{Fig_cross_tamari}.}
\end{figure}
    \begin{figure}[!h]
    \centering
\includegraphics[width=0.7\textwidth]{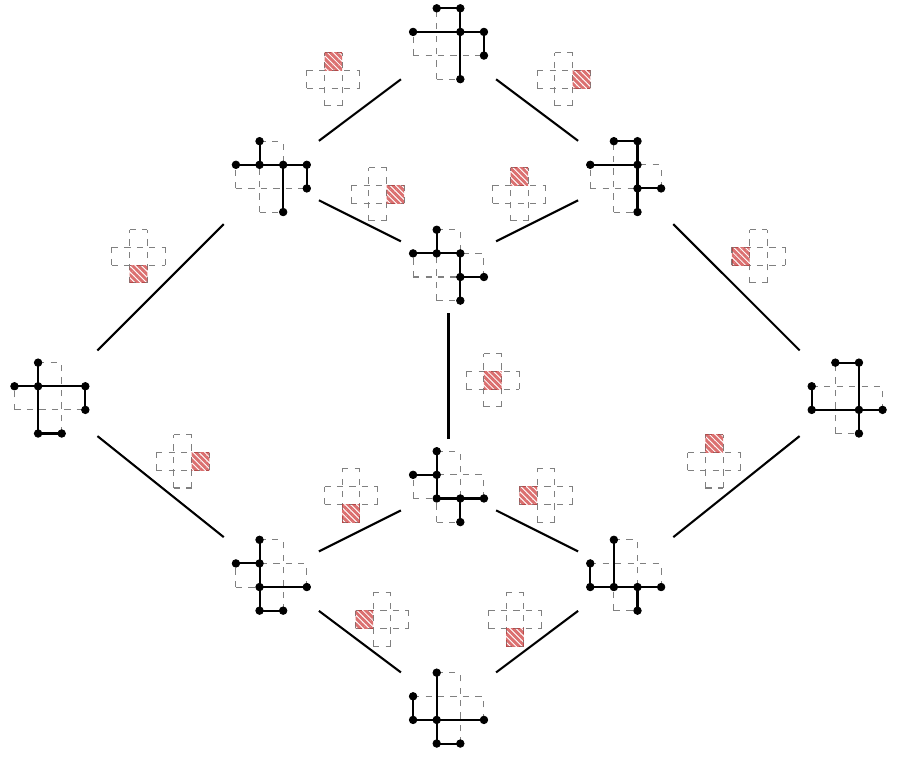}
    \caption{Cross Tamari lattice with a perspective edge-labeling via boxes.}
    \label{Fig_cross_tamari}
    \end{figure}
\end{example}

\begin{conjecture}\label{conjectutre_cross}
    The box complex realizes the canonical join complex of the cross Tamari lattice.
\end{conjecture}

The crucial step in proving Conjecture~\ref{conjectutre_cross} is to generalize the edge-labeling established for alt $\nu$-Tamari lattices, see Definition \ref{definition_ji_tree}. Consider, for instance, \Cref{Fig_cross_tamari}, where the edges incident to the rightmost tree are not labeled by the bottom-right box. Consequently, Definition \ref{definition_ji_tree} does not yield the perspective edge-labeling for cross Tamari lattices. To address this limitation, we propose a refined definition that either incorporates supplementary rules or utilizes a broader framework based on framing lattices.

\begin{remark}
    Assuming Conjecture~\ref{conjectutre_cross} holds, the proof of vertex decomposability directly extends to cross Tamari lattices. %Moreover, Proposition \ref{theorem_reciprocity} and \cref{theorem_wedge} generalize to invariance under column permutations.
\end{remark}
        
%\begin{remark}
%    Lemma~\ref{Lemma_f_vector} establishes that for alt $\nu$-Tamari lattices, the $f$-vector of the canonical join complex is invariant under column permutations. This invariance also applies to row permutations. As a result, the $f$-vector of the canonical join complex for any cross Tamari lattice is invariant under row permutations. In particular, these complexes all share the same Euler characteristic.
%\end{remark}
    
\section{Summary}
In \Cref{section::boxcomplex}, we introduced the box complex $\Delta_u$ and established it as a combinatorial model for the canonical join complex of the alt $\nu$-Tamari lattice (\Cref{realization}). We used the box complex to show vertex decomposability (\cref{vertex_decomposable}) and shellability (Corollary~\ref{cor_shellable}). 

As another application of this new combinatorial model, we demonstrated that, for a fixed path $\nu$, the Euler characteristic of the canonical join complex of any alt~$\nu$-Tamari lattice is determined by the $\nu$-Narayana polynomial and is independent of~$\delta$ (\Cref{theorem_reciprocity}). 
Finally, we studied the homotopy type of the canonical join complex and showed that the number of top-dimensional spheres is invariant across all alt $\nu$-Tamari lattices (\Cref{thm_hf_invariance}). In the special case where $\nu_i \geq 2$ for all $i$, this number is equal to the number of shrunken Dyck paths (\Cref{theorem_wedge}).
Finally, although alt $\nu$-Tamari lattices are closely related, we established that their canonical join complexes are not homotopy equivalent.

%We introduced in \Cref{section::boxcomplex}, the box complex $\Delta_u$ and have established, it is a combinatorial model for the canonical join complex of the alt $\nu$-Tamari lattice, \Cref{realization}. This complex $\Delta_u$ is vertex decomposable. As an application of our new model, we demonstrated, that the Euler characteristic of the canonical join complex of the alt $\nu$-Tamari $Tam_{\delta,\nu}$ lattice is constant and given by the Narayana polynomial evaluated at $-1$. This is a new combinatorial reciprocity. As a consequence that the box complex is vertex decomposable, we obtain that the canonical join complex of the alt $\nu$-Tamari lattice is shellable. Therefore it is homotopically equivalent to a wedge of spheres. Considering this spheres, it turns out, that the number of highest dimensional spheres can be computed easily. For the $m$-Tamari $Tam_m^{n+1}$ and $m$-Dyck lattice $Dyck_m^{n+1}$ we have seen, that the higest dimensional spheres is of dimension $n$ and counted by Fuss catalan numbers. Even the Dyck and Tamari lattices seem to be really related, it turns out, that their canonical join complexes are not homotopically equivalent.

\section*{Acknowledgement}
This paper was partially supported by the Austrian Science Fund FWF, grant P~$33278$.

%In particular,  and Vincent Pilaud for the brief exchange at the 2025 BIRS workshop on lattice theory.

%This paper grew out of the 2025 BIRS workshop on lattice theory. We are grateful to Osamu Iyama and Nathan Williams for co-organizing the workshop, and we thank the other participants for stimulating conversations.

%related to this work, and in particular David Speyer for introducing me to this line of ideas.

%This work was funded by the Austrian Science Fund (FWF), project numbers  10.55776/P33278 and 10.55776/I5788.

%BIBLIOGRAPHY
% You do not have to use the same format for your references, but 
%    include everything in this file.
% If you use BibTeX to create a bibliography, copy the .bbl file into here.
% We recommend you use \doi{...} and \arxiv{...} like the examples below,
% as they give a short display form with an active link to the full url.

%\begin{thebibliography}{99}
%\bibitem{Bollobas} B. Bollob{\'a}s.  Almost every  graph has reconstruction number three.  \emph{J. Graph Theory},  14(1): 1--4, 1990.
%\end{thebibliography}

%arxiv:
\bibliographystyle{alpha}
\bibliography{sn-bibliography}  

@incollection{wachs2007poset,
  author    = {Wachs, Michelle L.},
  title     = {Poset topology: tools and applications},
  booktitle = {Geometric Combinatorics},
  series    = {IAS/Park City Mathematics Series},
  volume    = {13},
  publisher = {American Mathematical Society},
  address   = {Providence, RI},
  year      = {2007},
  pages     = {497--615}
}

@misc{axelrodfreed2025chute,
    AUTHOR={Ilani Axelrod-Freed and Colin Defant and Hanna Mularczyk and Son Nguyen and Katherine Tung},
    TITLE={Chute Move Posets are Lattices},
    YEAR={2025},
    HOWPUBLISHED={\arxiv{2507.13214}}
}

@misc{billey2025lattice,
    AUTHOR={Sara C. Billey and Connor McCausland and Clare Minnerath},
    TITLE={A Proof of {R}ubey's Lattice Conjecture},
    YEAR={2025},
    HOWPUBLISHED={\arxiv{2507.18852}}
}

@article {rubey2012maximal,
    AUTHOR = {Rubey, Martin},
     TITLE = {Maximal {$0$}-{$1$}-fillings of moon polyominoes with
              restricted chain lengths and rc-graphs},
   JOURNAL = {Adv. in Appl. Math.},
  FJOURNAL = {Advances in Applied Mathematics},
    VOLUME = {48},
      YEAR = {2012},
    NUMBER = {2},
     PAGES = {290--305},
      ISSN = {0196-8858,1090-2074},
   MRCLASS = {05B50 (05A19 05E05 06A07)},
  MRNUMBER = {2873877},
MRREVIEWER = {Svetlana\ Poznanovi\'c},
       DOI = {10.1016/j.aam.2011.05.005},
       URL = {https://doi.org/10.1016/j.aam.2011.05.005},
}

@article{jonsson2005,
    AUTHOR = {Jonsson, Jakob},
     TITLE = {Generalized triangulations and diagonal-free subsets of stack
              polyominoes},
   JOURNAL = {J. Combin. Theory Ser. A},
  FJOURNAL = {Journal of Combinatorial Theory. Series A},
    VOLUME = {112},
      YEAR = {2005},
    NUMBER = {1},
     PAGES = {117--142},
      ISSN = {0097-3165},
     CODEN = {JCBTA7},
   MRCLASS = {05A15 (05B50)},
  MRNUMBER = {2167478 (2006d:05011)},
MRREVIEWER = {Eric S. Egge},
       DOI = {10.1016/j.jcta.2005.01.009},
       FURL = {http://dx.doi.org/10.1016/j.jcta.2005.01.009},
}

@misc{vonbell2025framing,
    AUTHOR={Matias von Bell and Cesar Ceballos},
    TITLE={Framing Lattices and Flow Polytopes},
    YEAR={2025},
    HOWPUBLISHED={\arxiv{2512.20575}}
}

@article {Wachs1999,
    AUTHOR = {Wachs, M. L.},
     TITLE = {Obstructions to shellability},
   JOURNAL = {Discrete Comput. Geom.},
  FJOURNAL = {Discrete \& Computational Geometry. An International Journal
              of Mathematics and Computer Science},
    VOLUME = {22},
      YEAR = {1999},
    NUMBER = {1},
     PAGES = {95--103},
      ISSN = {0179-5376,1432-0444},
   MRCLASS = {52B22 (05E25 06A07)},
  MRNUMBER = {1692690},
MRREVIEWER = {Eva-Maria\ E.\ Feichtner},
       DOI = {10.1007/PL00009450},
       URL = {https://doi.org/10.1007/PL00009450},
}

@book{Jonsson2008,
  author    = {Jonsson, Jakob},
  title     = {Simplicial Complexes of Graphs},
  series    = {Lecture Notes in Mathematics},
  publisher = {Springer Verlag, Berlin},
  year      = {2008},
  doi       = {10.1007/978-3-540-75859-4},
  isbn      = {978-3-540-75859-4}
}

@book {Tamari51,
    AUTHOR = {Tamari, Dov},
     TITLE = {Mono\"ides pr\'eordonn\'es et cha\^ines de {M}alcev},
      NOTE = {Th\`ese},
 PUBLISHER = {Universit\'e{} de Paris, Paris},
      YEAR = {1951},
     PAGES = {iv+81 pp. (mimeographed)},
   MRCLASS = {20.0X},
  MRNUMBER = {51833},
MRREVIEWER = {D.\ C.\ Murdoch},
}

@article{Reading2004,
  author  = {Reading, Nathan},
  title   = {Lattice congruences of the weak order},
  journal = {Order},
  year    = {2004},
  volume  = {21},
  number  = {4},
  pages   = {315--344},
  doi     = {10.1007/s11083-005-4803-8}
}

@article {Anders1996,
    AUTHOR = {Bj\"orner, Anders and Wachs, Michelle L.},
     TITLE = {Shellable nonpure complexes and posets. {I}},
   JOURNAL = {Trans. Amer. Math. Soc.},
  FJOURNAL = {Transactions of the American Mathematical Society},
    VOLUME = {348},
      YEAR = {1996},
    NUMBER = {4},
     PAGES = {1299--1327},
      ISSN = {0002-9947,1088-6850},
   MRCLASS = {06A08 (05E99 52B99)},
  MRNUMBER = {1333388},
MRREVIEWER = {T.\ S.\ Blyth},
       DOI = {10.1090/S0002-9947-96-01534-6},
       URL = {https://doi.org/10.1090/S0002-9947-96-01534-6},
}

@article {Bergeron2012,
    AUTHOR = {Bergeron, Fran\c cois and Pr\'eville-Ratelle, Louis-Fran\c
              cois},
     TITLE = {Higher trivariate diagonal harmonics via generalized {T}amari
              posets},
   JOURNAL = {J. Comb.},
  FJOURNAL = {Journal of Combinatorics},
    VOLUME = {3},
      YEAR = {2012},
    NUMBER = {3},
     PAGES = {317--341},
      ISSN = {2156-3527,2150-959X},
   MRCLASS = {05E10 (05A19)},
  MRNUMBER = {3029440},
MRREVIEWER = {Anthony\ A.\ Mendes},
       DOI = {10.4310/JOC.2012.v3.n3.a4},
       URL = {https://doi.org/10.4310/JOC.2012.v3.n3.a4},
}

@article {Sanley1975,
    AUTHOR = {Stanley, Richard P.},
     TITLE = {The {F}ibonacci lattice},
   JOURNAL = {Fibonacci Quart.},
  FJOURNAL = {The Fibonacci Quarterly. Official Organ of the Fibonacci
              Association},
    VOLUME = {13},
      YEAR = {1975},
    NUMBER = {3},
     PAGES = {215--232},
      ISSN = {0015-0517},
   MRCLASS = {06A35},
  MRNUMBER = {387143},
MRREVIEWER = {Joel\ Berman},
}

@article {Tamari1962,
    AUTHOR = {Tamari, Dov},
     TITLE = {The algebra of bracketings and their enumeration},
   JOURNAL = {Nieuw Arch. Wisk. (3)},
  FJOURNAL = {Nieuw Archief voor Wiskunde. Derde Serie},
    VOLUME = {10},
      YEAR = {1962},
     PAGES = {131--146},
      ISSN = {0028-9825},
   MRCLASS = {17.10},
  MRNUMBER = {146227},
MRREVIEWER = {R.\ H.\ Oehmke},
}

@article {GN2010,
    AUTHOR = {Gr\"atzer, G. and Nation, J. B.},
     TITLE = {A new look at the {J}ordan-{H}\"older theorem for semimodular
              lattices},
   JOURNAL = {Algebra Universalis},
  FJOURNAL = {Algebra Universalis},
    VOLUME = {64},
      YEAR = {2010},
    NUMBER = {3-4},
     PAGES = {309--311},
      ISSN = {0002-5240,1420-8911},
   MRCLASS = {06C10 (06B05 06B10)},
  MRNUMBER = {2781081},
MRREVIEWER = {Ulrich\ Faigle},
       DOI = {10.1007/s00012-011-0104-9},
       URL = {https://doi.org/10.1007/s00012-011-0104-9},
}

@article{JJ1954,
author = {Jakubík, Ján},
journal = {Časopis pro pěstování matematiky},
language = {slo},
number = {2},
pages = {206-216},
publisher = {Mathematical Institute of the Czechoslovak Academy of Sciences},
title = {Relácie kongruentnosti a slabá projektívnosť vo sväzoch},
url = {http://eudml.org/doc/18999},
volume = {080},
year = {1955},
}

@article {Anders1997,
    AUTHOR = {Bj\"orner, Anders and Wachs, Michelle L.},
     TITLE = {Shellable nonpure complexes and posets. {II}},
   JOURNAL = {Trans. Amer. Math. Soc.},
  FJOURNAL = {Transactions of the American Mathematical Society},
    VOLUME = {349},
      YEAR = {1997},
    NUMBER = {10},
     PAGES = {3945--3975},
      ISSN = {0002-9947,1088-6850},
   MRCLASS = {06A08 (05E99)},
  MRNUMBER = {1401765},
MRREVIEWER = {Volkmar\ Welker},
       DOI = {10.1090/S0002-9947-97-01838-2},
       URL = {https://doi.org/10.1090/S0002-9947-97-01838-2},
}

@article {CCh2024,
    AUTHOR = {Ceballos, Cesar and Chenevi\`ere, Cl\'ement},
     TITLE = {On linear intervals in the alt {$\nu$}-{T}amari lattices},
   JOURNAL = {Comb. Theory},
  FJOURNAL = {Combinatorial Theory},
    VOLUME = {4},
      YEAR = {2024},
    NUMBER = {2},
     PAGES = {Paper No. 18, 31},
      ISSN = {2766-1334},
   MRCLASS = {06A07 (05A19 06B05)},
  MRNUMBER = {4807157},
MRREVIEWER = {Joel\ Berman},
}

@article {Provan1980,
    AUTHOR = {Provan, J. Scott and Billera, Louis J.},
     TITLE = {Decompositions of simplicial complexes related to diameters of
              convex polyhedra},
   JOURNAL = {Math. Oper. Res.},
  FJOURNAL = {Mathematics of Operations Research},
    VOLUME = {5},
      YEAR = {1980},
    NUMBER = {4},
     PAGES = {576--594},
      ISSN = {0364-765X,1526-5471},
   MRCLASS = {52A25 (90C05)},
  MRNUMBER = {593648},
MRREVIEWER = {J.\ Parida},
       DOI = {10.1287/moor.5.4.576},
       URL = {https://doi.org/10.1287/moor.5.4.576},
}

@article {Barnard2020,
    AUTHOR = {Barnard, Emily},
     TITLE = {The canonical join complex of the {T}amari lattice},
   JOURNAL = {J. Combin. Theory Ser. A},
  FJOURNAL = {Journal of Combinatorial Theory. Series A},
    VOLUME = {174},
      YEAR = {2020},
     PAGES = {105207, 30},
      ISSN = {0097-3165,1096-0899},
   MRCLASS = {05E45},
  MRNUMBER = {4082060},
       DOI = {10.1016/j.jcta.2019.105207},
       URL = {https://doi.org/10.1016/j.jcta.2019.105207},
}

@article {Muehle2021,
    AUTHOR = {M\"uhle, Henri},
     TITLE = {Noncrossing arc diagrams, {T}amari lattices, and parabolic
              quotients of the symmetric group},
   JOURNAL = {Ann. Comb.},
  FJOURNAL = {Annals of Combinatorics},
    VOLUME = {25},
      YEAR = {2021},
    NUMBER = {2},
     PAGES = {307--344},
      ISSN = {0218-0006,0219-3094},
   MRCLASS = {06B05 (05E16 06B10)},
  MRNUMBER = {4268292},
MRREVIEWER = {Konrad\ P.\ Pi\'oro},
       DOI = {10.1007/s00026-021-00532-9},
       URL = {https://doi.org/10.1007/s00026-021-00532-9},
}

@article {Muehle2023,
    AUTHOR = {M\"uhle, Henri},
     TITLE = {Meet-distributive lattices have the intersection property},
   JOURNAL = {Math. Bohem.},
  FJOURNAL = {Academy of Sciences of the Czech Republic. Mathematical
              Institute. Mathematica Bohemica},
    VOLUME = {148},
      YEAR = {2023},
    NUMBER = {1},
     PAGES = {95--104},
      ISSN = {0862-7959,2464-7136},
   MRCLASS = {06D75},
  MRNUMBER = {4536312},
MRREVIEWER = {Himadri\ Mukherjee},
}

@article {Adaricheva2003,
    AUTHOR = {Adaricheva, K. V. and Gorbunov, V. A. and Tumanov, V. I.},
     TITLE = {Join-semidistributive lattices and convex geometries},
   JOURNAL = {Adv. Math.},
  FJOURNAL = {Advances in Mathematics},
    VOLUME = {173},
      YEAR = {2003},
    NUMBER = {1},
     PAGES = {1--49},
      ISSN = {0001-8708,1090-2082},
   MRCLASS = {06B05 (06B15 08C15 51D99)},
  MRNUMBER = {1954454},
MRREVIEWER = {Ivo\ D\"untsch},
       DOI = {10.1016/S0001-8708(02)00011-7},
       URL = {https://doi.org/10.1016/S0001-8708(02)00011-7},
}

@book {Freese1995,
    AUTHOR = {Freese, Ralph and Je\v{z}ek, Jaroslav and Nation, James B.},
     TITLE = {Free lattices},
    SERIES = {Mathematical Surveys and Monographs},
    VOLUME = {42},
 PUBLISHER = {American Mathematical Society, Providence, RI},
      YEAR = {1995},
     PAGES = {viii+293},
      ISBN = {0-8218-0389-1},
   MRCLASS = {06B25 (06-02 06-04 06B20 68Q25)},
  MRNUMBER = {1319815},
MRREVIEWER = {T.\ S.\ Blyth},
       DOI = {10.1090/surv/042},
       URL = {https://doi.org/10.1090/surv/042},
}

@article {reading2015,
    AUTHOR = {Reading, Nathan},
     TITLE = {Noncrossing arc diagrams and canonical join representations},
   JOURNAL = {SIAM J. Discrete Math.},
  FJOURNAL = {SIAM Journal on Discrete Mathematics},
    VOLUME = {29},
      YEAR = {2015},
    NUMBER = {2},
     PAGES = {736--750},
      ISSN = {0895-4801,1095-7146},
   MRCLASS = {05A05 (05E15 06B10)},
  MRNUMBER = {3335492},
MRREVIEWER = {Paula\ M. Machado Cruz Catarino},
       DOI = {10.1137/140972391},
       URL = {https://doi.org/10.1137/140972391},
}

@article {barnard2019,
    AUTHOR = {Barnard, Emily},
     TITLE = {The canonical join complex},
   JOURNAL = {Electron. J. Combin.},
  FJOURNAL = {Electronic Journal of Combinatorics},
    VOLUME = {26},
      YEAR = {2019},
    NUMBER = {1},
     PAGES = {Paper No. 1.24, 25},
      ISSN = {1077-8926},
   MRCLASS = {05E45 (05E10 06A07)},
  MRNUMBER = {3919619},
MRREVIEWER = {Konrad\ P.\ Pi\'oro},
       DOI = {10.37236/7866},
       URL = {https://doi.org/10.37236/7866},
}

@article {preville_vTamari_2017,
    AUTHOR = {Pr\'eville-Ratelle, Louis-Fran\c and Viennot, Xavier},
     TITLE = {The enumeration of generalized {T}amari intervals},
   JOURNAL = {Trans. Amer. Math. Soc.},
  FJOURNAL = {Transactions of the American Mathematical Society},
    VOLUME = {369},
      YEAR = {2017},
    NUMBER = {7},
     PAGES = {5219--5239},
      ISSN = {0002-9947,1088-6850},
       DOI = {10.1090/tran/7004},
       URL = {https://doi.org/10.1090/tran/7004},
}

@article {ceballos_vTamari_subword_2020,
    AUTHOR = {Ceballos, Cesar and Padrol, Arnau and Sarmiento, Camilo},
     TITLE = {The {$\nu$}-{T}amari lattice via {$\nu$}-trees,
              {$\nu$}-bracket vectors, and subword complexes},
   JOURNAL = {Electron. J. Combin.},
  FJOURNAL = {Electronic Journal of Combinatorics},
    VOLUME = {27},
      YEAR = {2020},
    NUMBER = {1},
     PAGES = {Paper No. 1.14, 31},
      ISSN = {1077-8926},
       DOI = {10.37236/8000},
       URL = {https://doi.org/10.37236/8000},
}

@misc{rowmotion,
    AUTHOR={Ben Adenbaum and Emily Barnard and Cesar Ceballos and Clément Chenevière and Colin Defant and Sam Hopkins and Matthias Müller and Martin Rubey and Jessica Striker},
    TITLE={Invariance of rowmotion for variants of the Tamari lattice},
    YEAR={2026},
    HOWPUBLISHED={\arxiv{2609.12983}}
}

@article {tamcom,
    AUTHOR = {Ceballos, Cesar and Padrol, Arnau and Sarmiento, Camilo},
     TITLE = {Geometry of {$\nu$}-{T}amari lattices in types {$A$} and~{$B$}},
   JOURNAL = {Trans. Amer. Math. Soc.},
  FJOURNAL = {Transactions of the American Mathematical Society},
    VOLUME = {371},
      YEAR = {2019},
    NUMBER = {4},
     PAGES = {2575--2622},
      ISSN = {0002-9947,1088-6850},
   MRCLASS = {05E45 (05E10 14T05 52B22)},
  MRNUMBER = {3896090},
MRREVIEWER = {Evgeny\ Smirnov},
       DOI = {10.1090/tran/7405},
       URL = {https://doi.org/10.1090/tran/7405},
}

@unpublished{alt,  
  author = {Cesar Ceballos},
  title = {A canonical realization of the alt $\nu$-associahedron},
  year = {2024},
  note = {\href{https://arxiv.org/abs/2401.17204v1}{arXiv:2401.17204v1}}
}
%\bibliography{arxiv.bbl}
\end{document}